\documentclass[11pt,leqno]{article}

\usepackage{amsthm,amsfonts,amssymb,amsmath,oldgerm}
\usepackage{fullpage}
\usepackage{graphicx}
\usepackage{mathrsfs}
\usepackage{xcolor}
\usepackage[colorinlistoftodos]{todonotes} 
\usepackage[hidelinks]{hyperref}

\numberwithin{equation}{section}

\newcommand{\R}{{\mathbb R}}\newcommand{\N}{{\mathbb N}}

\newcommand{\C}{{\mathbb C}}

\newcommand{\re}{\mathrm{e}}
\newcommand{\de}{\mathrm{d}}
\newcommand{\ri}{\mathrm{i}}
\renewcommand{\Re}{\operatorname{Re}}
\newcommand{\curlN}{\mathcal{N}}

\newcommand{\curlA}{\mathcal{A}}

\newcommand{\curlG}{\mathcal{G}}
\newcommand{\snorm}[2][]{\ensuremath{\left\vert#2\right\vert_{#1}}}
\newcommand{\norm}[2][]{\ensuremath{\left\|#2\right\|_{#1}}}

\let\epsilon\varepsilon
\let\theta\vartheta

\newtheorem{theorem}{Theorem}[section]\newtheorem{lemma}[theorem]{Lemma}

\newtheorem{proposition}[theorem]{Proposition}
\newtheorem{corollary}[theorem]{Corollary}

\newtheorem{remark}[theorem]{Remark}
\allowdisplaybreaks[3]

\title{Nonlinear stability of periodic roll solutions in the real Ginzburg-Landau equation against $C_{\mathrm{ub}}^m$-perturbations}

\author{Bastian Hilder\thanks{Centre for Mathematical Sciences, Lund University, PO Box 118, 221 00 Lund, Sweden; \texttt{bastian.hilder@math.lu.se}}
\qquad Bj\"orn de Rijk\thanks{Karlsruhe Institute of Technology, Englerstra\ss e 2, 76131 Karlsruhe, Germany;  \texttt{bjoern.de-rijk@kit.edu}}
\qquad Guido Schneider\thanks{Institut f\"ur Analysis, Dynamik und Modellierung, Universit\"at Stuttgart, Pfaffenwaldring 57, 70569 Stuttgart, Germany; \texttt{guido.schneider@mathematik.uni-stuttgart.de}}
}

\begin{document}

\maketitle
\begin{abstract}
The real Ginzburg-Landau equation arises as a universal amplitude equation for the description of pattern-forming systems exhibiting a Turing bifurcation. It possesses spatially periodic roll solutions which are known to be stable against localized perturbations. It is the purpose of this paper to prove their stability against bounded perturbations, which are not necessarily localized. Since all state-of-the-art techniques rely on localization or periodicity properties of perturbations, we develop a new method, which employs pure $L^\infty$-estimates only. By fully exploiting the smoothing properties of the semigroup generated by the linearization, we are able to close the nonlinear iteration despite the slower decay rates. To show the wider relevance of our method, we also apply it to the amplitude equation as it appears for pattern-forming systems with an additional conservation law.
\end{abstract}

\section{Introduction}

The nonlinear stability of periodic waves against localized perturbations in spatially extended systems had been an open problem for several decades. The main difficulty lies in the fact that the linearization of such a system about the periodic wave has, when posed on any natural space of localized functions, continuous spectrum, which touches the origin due to translational invariance. Thus, the semigroup generated by the linearization exhibits algebraic decay rates at best, which heavily complicates the nonlinear analysis. This contrasts with the case of co-periodic perturbations, where the linearization has discrete spectrum, so that standard orbital stability techniques apply leading to exponential decay rates.

The longstanding question of nonlinear stability against localized perturbations was first resolved in~\cite{CEE92} for periodic roll solutions
\begin{equation}\label{rolls}
\sqrt{1-q^2} \re^{\ri qx},
\end{equation}
with wavenumber $q \in (-1,1)$, in the real Ginzburg-Landau equation
\begin{equation}\label{GL}
\partial_t A = \partial_x^2 A + A - A|A|^2, \qquad A(x,t) \in \C,
\end{equation}
with $x \in \R$ and $t\geq 0$, which arises as a universal amplitude equation for the description of pattern-forming systems close to a Turing bifurcation, see~\cite{SU17book} and references therein. In the nonlinear stability analysis in~\cite{CEE92} one takes $q^2 < \frac{1}{3}$ so that the spectrum of the linearization of~\eqref{GL} about the periodic roll solution~\eqref{rolls} lies in the open left-half plane except for a parabolic touching with the origin. Upon crossing the so-called Eckhaus boundary at $q^2 = \frac{1}{3}$ the periodic roll solution undergoes a sideband destabilization, see Remark~\ref{rem:Eckhaus}. The analysis in~\cite{CEE92} is based on iterative estimates on the associated Duhamel formula, leading to nonlinear stability of the periodic rolls against localized perturbations for $q^2 < \frac{1}{3}$. Using renormalization techniques, the result in~\cite{CEE92} was later extended in~\cite{BK92,GM98} to prove stable diffusive mixing of the asymptotic states
\begin{align} \label{mixing} \sqrt{1-q_{\pm}^2}\, \re^{\ri q_{\pm}x + \beta_{\pm}},\end{align}
with different asymptotic wavenumbers $q_\pm^2 < \frac{1}{3}$ and phases $\beta_\pm \in \R$ for solutions that initially converge to these states as $x \to \pm \infty$.

The stability analyses in~\cite{BK92,CEE92,GM98} all rely on the principle that a sufficiently smooth nonlinearity improves localization, e.g.~if a function $u$ is $L^2$-localized then its square $u^2$ is $L^1$-localized. The gained localization can then be used to pick up algebraic decay from the semigroup, e.g.~the heat semigroup $\re^{\partial_x^2 t}$ decays at rate $t^{-1/4}$ as an operator from $L^1(\R)$ into $L^2(\R)$. We refer to~\cite[Section 3.1]{MSU01} and~\cite[Section~14.1.3]{SU17book} for an illustration of how this principle can be employed to close the nonlinear iteration.

In this paper we further extend the stability theory of periodic rolls in the real Ginzburg-Landau equation~\eqref{GL} by considering perturbations, which do not exhibit any form of localization or periodicity. Consequently, we cannot rely on the aforementioned principle of localization-induced decay. Instead, we present a completely new scheme, which relies on pure $L^\infty$-estimates only and fully exploits the smoothing properties of the semigroup, e.g.~the \emph{derivative} $\partial_x \re^{\partial_x^2 t}$ of the heat semigroup decays on $L^\infty(\R)$ with rate $t^{-\frac{1}{2}}$. Consequently, derivatives of bounded perturbations can be expected to decay. Since it turns out that the most critical terms in the nonlinearity of the perturbation equation contain derivatives, we are able to close the nonlinear iteration despite slower decay rates due to loss of localization-induced decay. We refer to~\S\ref{sec:ill} for an illustration of the main ideas of our pure $L^\infty$-scheme in a simple setting.

Thus, we establish nonlinear stability of the periodic solutions~\eqref{rolls} of the real Ginzburg-Landau equation~\eqref{GL} against bounded, sufficiently smooth perturbations. More precisely, the perturbations lie in the space $C_{\mathrm{ub}}^2(\R)$, where $C_{\mathrm{ub}}^m(\R)$, $m \in \mathbb{N}_0$ denotes the space of bounded and uniformly continuous functions, which are $m$ times differentiable and whose $m$ derivatives are also bounded and uniformly continuous. All in all, we establish the following result.

\begin{theorem}\label{thm:stabGL_intro}
Let $q^2 < 1/3$. Then, there exist $M_0, \varepsilon_0 > 0$ such that for all $\varepsilon \in (0,\varepsilon_0)$ and $v_0 \in C^2_{\mathrm{ub}}(\R,\C)$ with $\norm[W^{2,\infty}]{v_0} < \varepsilon$ there exists a global classical solution 
$$A \in  C^1\big([0,\infty),C_{\mathrm{ub}}(\R,\C)\big) \cap C\big([0,\infty), C^2_{\mathrm{ub}}(\R,\C)\big),$$ 
of the real Ginzburg-Landau equation~\eqref{GL} with initial condition
$$A(x,0) = \sqrt{1-q^2} \,\re^{\ri qx} + v_0(x),$$
satisfying
\begin{align*}
\sup_{x \in \R} \left|A(x,t) - \sqrt{1-q^2}\, \re^{\ri q x}\right| \leq M_0 \varepsilon,
\end{align*}
for all $t \geq 0$.
\end{theorem}

The proof of Theorem~\ref{thm:stabGL_intro} is a direct consequence of the upcoming Theorem~\ref{thm:stabmodGL}, as will be explained below.

In contrast to the stability result against localized perturbations in~\cite{CEE92}, the perturbation in Theorem~\ref{thm:stabGL_intro} does not decay in $L^\infty(\R)$ as $t \to \infty$, i.e.~we do not obtain \emph{asymptotic} nonlinear stability. Indeed, the initial perturbation $v_0(x) = \sqrt{1-q^2} \,\re^{\ri q x}\left(\re^{\ri \theta} - 1\right)$, whose $W^{2,\infty}$-norm can be taken arbitrarily small by taking $\theta \in (-\pi,\pi] \setminus \{0\}$ sufficiently close to $0$, yields a spatial translate of the periodic wave which, being a solution of~\eqref{GL} itself, does not converge to~\eqref{rolls} in $L^\infty(\R)$ as $t \to \infty$. In fact, similar behavior arises in the simpler heat equation $\partial_t u = \partial_x^2 u$. Here, solutions with small initial data in $L^\infty(\R)$ stay small, but do not necessarily decay, which can readily be seen by considering constant solutions. However, derivatives of solutions of the heat equation with initial conditions in $L^\infty(\R)$ do decay due to the smoothing properties of the associated semigroup. This can also be observed in the real Ginzburg-Landau equation~\eqref{GL}. By writing the perturbed solution in Theorem~\ref{thm:stabGL_intro} in polar form
\begin{align}A(x,t) = \sqrt{1-q^2}\,\re^{\ri q x} \re^{r(x,t) + \ri \phi(x,t)}, \label{eq:polar}\end{align}
with $r(x,t)$ and $\phi(x,t)$ real-valued functions, one can show that $r(x,t)$ and $\partial_x \phi(x,t)$ decay in $L^\infty(\R)$. We refer to~\S\ref{sec:GL} for the precise statement. 

To our best knowledge, we are not aware of any nonlinear stability result of periodic waves in spatially extended systems against $C_{\mathrm{ub}}^m$-perturbations. That is, we believe that Theorem~\ref{thm:stabGL_intro} goes beyond the current state-of-the-art by lifting any localization or periodicity requirement on perturbations, see Remark~\ref{rem:new}. Moreover, due to the role of the real Ginzburg-Landau equation as a universal amplitude equation, we strongly expect that our $L^\infty$-scheme can be extended to handle the nonlinear stability of periodic waves against $C_{\mathrm{ub}}^m$-perturbations in various pattern-forming and hydrodynamical systems. With this regard we emphasize that the nonlinear stability results~\cite{BK92,CEE92,GM98} in the real Ginzburg-Landau equation~\eqref{GL} against localized perturbations have also been extended to numerous systems such as the Swift-Hohenberg equation~\cite{Schn96}, the Taylor-Couette problem~\cite{Schn98}, the inclined film problem~\cite{UEC} and general reaction-diffusion systems~\cite{JONZW,JONZ,SAN3,Schn97}.

\begin{remark} \label{rem:new}
{\rm
Although the authors are not aware of any nonlinear stability result for periodic waves in spatially extended systems against $C_{\mathrm{ub}}^m$-perturbations, nonlocalized modulations of phase or wavenumber and partially nonlocalized perturbations have been considered in the literature~\cite{BK92,RS18,GM98,IYSA,JONZW,SAN3}. However, all of these result still crucially rely on localization-induced decay. Indeed, although the planar perturbations in~\cite{RS18} are nonlocalized along a line in $\R^2$, they are required to decay exponentially in distance from that line. Moreover, the initial phase or wavenumber off-set of the modulated periodic solutions can be nonlocalized in~\cite{BK92,GM98,IYSA,JONZ,SAN3}, but it must converge sufficiently fast to asymptotic limits as $x \to \pm \infty$ yielding localization of its derivative.
}
\end{remark}

\begin{remark} \label{rem:Eckhaus}
{\rm
Nonlinear stability against localized perturbations of periodic rolls~\eqref{rolls} in the real Ginzburg-Landau equation~\eqref{GL} exactly at the Eckhaus boundary, i.e.~at $q^2 = \frac{1}{3}$, has been established in~\cite{GSWZ}. At $q^2 = \frac{1}{3}$ the spectrum of the linearization is at the threshold of a sideband destabilization and thus touches the origin in a quartic tangency. Hence, the associated semigroup exhibits algebraic decay at rate $t^{-1/4}$ as an operator of $L^1(\R)$ into $L^\infty(\R)$, as opposed to the decay at rate $t^{-1/2}$ which arises for $q^2 < \frac{1}{3}$. The fact that the nonlinear iteration can be closed despite weaker localization-induced decay indicates that it might not be crucial for the argument. This is confirmed by the findings in this paper, where nonlinear stability is established without exploiting localization-induced decay.
}\end{remark}

\subsection{Additional conservation law}

There are many physically interesting pattern-forming systems which exhibit a conservation law, such as the B\'enard-Marangoni problem~\cite{TAK}, the flow down an inclined plane~\cite{CD02}, or the Faraday experiment~\cite{ANR14}. For many such systems the dynamics close to a Turing bifurcation cannot be described by the real Ginzburg-Landau equation. Instead, the \emph{modified Ginzburg-Landau system}
\begin{align}
\label{eq:modGL}
\begin{split}
		\partial_t A &= \partial_x^2 A + A + AB - A \snorm{A}^2, \\
		\partial_t B &= D \partial_x^2 B + \gamma \partial_x^2(\snorm{A}^2),
\end{split} \qquad A(x,t) \in \C, \, B(x,t) \in \R,
\end{align}
with parameters $D > 0$ and $\gamma \in \R$, arises as a generic amplitude equation for such pattern-forming systems with an additional conservation law, cf.~\cite{HSZ11,MC00,SZ17}. Its simplest patterns are periodic rolls of the form
\begin{align}\left(\sqrt{1-q^2} \,\re^{\ri qx},0\right),\label{eq:solsmodGL}\end{align}
with wavenumber $q \in (-1,1)$, see also Remark~\ref{rem:2par}. The presence of an additional conservation law complicates the nonlinear stability analysis of the periodic rolls, since the spectrum of the linearization possesses multiple curves touching the imaginary axis at the origin. This yields multiple critical modes, whose interactions could obstruct a nonlinear stability argument. 

Nevertheless, nonlinear stability of the periodic rolls~\eqref{eq:solsmodGL} in the modified Ginzburg-Landau system~\eqref{eq:modGL} against localized mean-zero perturbations has been obtained in~\cite{SZ17}. Here, one assumes that the wavenumber and system parameters satisfy
\begin{align} D + \gamma - \frac{2Dq^2}{1 - q^2} > 0, \qquad q^2 < \frac{1}{3}, \label{e:spec}\end{align}
so that the periodic wave~\eqref{eq:solsmodGL} is spectrally stable, see the upcoming Lemma~\ref{lem:spec_cont} and Remark~\ref{rem:spectralStabilitygGL}. In addition to localized-induced decay, the proof in~\cite{SZ17} exploits that the semigroup generated by the linearization exhibits faster decay when applied to mean-zero perturbations. A similar effect occurs in the heat equation $\partial_t u = \partial_x^2 u$. Indeed, if $f \in L^1(\R)$ has mean zero, then $\re^{\partial_x^2 t} f$ exhibits decay at rate $t^{-1}$ in $L^\infty(\R)$, instead of the usual decay at rate $t^{-1/2}$.

In the special case $q = 0$, when \eqref{eq:solsmodGL} degenerates to a homogeneous solution of the modified Ginzburg-Landau system \eqref{eq:modGL}, the nonlinear stability of \eqref{eq:solsmodGL} against localized perturbations has been shown in \cite{Hilder21} without the restriction to mean-zero perturbations.
The proof exploits the special structure of the linearity for $q = 0$, see also Lemma \ref{lem:difest} and Theorem \ref{thm:global}.

In this paper we consider the behavior of the periodic solutions~\eqref{eq:solsmodGL} in the modified Ginzburg-Landau system~\eqref{eq:modGL} under $C_{\mathrm{ub}}^m$-perturbations. The slower decay rates in combination with multiple interacting critical modes makes it challenging to close a nonlinear iteration and provides an excellent test for the robustness of our newly developed $L^\infty$-scheme. The outcome of our analysis is that we can control small $C_{\mathrm{ub}}^m$-perturbations on exponentially long time scales. The reason that we do not acquire global control is that the linear decay on the $B$-component is just too slow to control all nonlinearities arising in the perturbation equations. Indeed, if the initial perturbation in the $B$-component is slightly localized, i.e.~$L^p$-localized for some $p \in [1,\infty)$, the linear decay in the $B$-component improves, so that we are able to close a global nonlinear argument. Moreover, a global result can be obtained in the special case $q = 0$. Then, the semigroup acting on the critical nonlinearity exhibits better decay rates. We refer to~\S\ref{sec:mainresults} for a full overview of the aforementioned results and state the relevant nonlinear stability result only.

\begin{theorem} \label{thm:stabmodGL}
Let $q, \gamma \in \R$ and $D > 0$ satisfy~\eqref{e:spec}. Then, there exist $M_0, \varepsilon_0 > 0$ such that for all $\varepsilon \in (0,\varepsilon_0)$ and 
$$(v_0,B_0) \in C^2_{\mathrm{ub}}(\R,\C) \times \left(C_{\mathrm{ub}}^1(\R) \cap L^1(\R)\right),$$ 
with $\norm[W^{2,\infty}]{v_0} + \norm[W^{1,\infty}]{B_0} + \norm[L^1]{B_0} < \varepsilon$ there exists a global mild solution 
$$\left(A,B\right) \in C\big([0,\infty), C^2_{\mathrm{ub}}(\R,\C)\big) \times C\big([0,\infty),C_{\mathrm{ub}}^1(\R)\big),$$ 
of the modified Ginzburg-Landau system~\eqref{eq:modGL} with initial condition
$$\left(A(x,0),B(x,0)\right) = \left(\sqrt{1-q^2} \,\re^{\ri qx} + v_0(x),B_0(x)\right),$$
satisfying
\begin{align*}
\sup_{x \in \R} \left[\left|A(x,t) - \sqrt{1-q^2} \,\re^{\ri q x}\right| + \sqrt{1+t}\,|B(x,t)|\right] \leq M_0 \varepsilon,
\end{align*}
for all $t \geq 0$. Moreover, if it holds in addition $(v_0,B_0) \in C^3_{\mathrm{ub}}(\R,\C) \times C_{\mathrm{ub}}^2(\R)$, then the solution $(A(t),B(t))$ is a classical solution of~\eqref{eq:modGL}.
\end{theorem}

The proof of Theorem~\ref{thm:stabmodGL} is a direct consequence of Theorem~\ref{thm:partloc} stated below.

We note that Theorem~\ref{thm:stabmodGL} improves the nonlinear stability result in~\cite{SZ17} by lifting the mean-zero condition in the $B$-component and removing the localization requirement in the $A$-component. Moreover, one readily observes that Theorem~\ref{thm:stabmodGL} implies Theorem~\ref{thm:stabGL_intro} by setting $\gamma = 0$ and $B_0 = 0$. Thus, from the discussion below Theorem~\ref{thm:stabGL_intro} we conclude that the presented decay rate in Theorem~\ref{thm:stabmodGL} on the $A$-component is sharp. On the other hand, for $\gamma = 0$ the $B$-component of~\eqref{eq:modGL} reduces to the linear heat equation, for which the presented decay at rate $t^{-1/2}$ is sharp.

It is interesting to note that the localization-induced decay is only exploited in the \emph{linear} estimates in the proof of Theorem~\ref{thm:stabmodGL}. That is, the nonlinear iteration is still based on a pure $L^\infty$-scheme. In fact, this is also a necessity: since $v_0$ is not $L^1$-localized, the $L^1$-norm of the perturbation in the $A$-component cannot be controlled and, thus, the $L^1$-norm of the nonlinearities in the $B$-component in~\eqref{eq:modGL} cannot be estimated, which prohibits \emph{nonlinear} $L^1$-estimates on the $B$-component. We refer to~\S\ref{sec:ill} for a simple example, where localization-induced decay is used in the linear estimates only and the nonlinear argument is purely based on $L^\infty$-estimates.

\begin{remark}\label{rem:2par}
{\rm
In fact, the modified Ginzburg-Landau system~\eqref{eq:modGL} possesses a 2-parameter family of periodic solutions
\begin{align}\left(\sqrt{1+b-q^2}\, \re^{\ri qx},b\right), \label{eq:solsmodGL2}\end{align}
with $q,b \in \R$ satisfying $1+b-q^2 > 0$. By introducing the new coordinate $\smash{\widetilde{B}} = B - b$ we transform the equilibrium  $(A,B) = (0,b)$ of the $(A,B)$-system~\eqref{eq:modGL} into the equilibrium $\smash{ ( A,\widetilde{B}) = (0,0) }$ of the $\smash{ (A,\widetilde{B})}$-system. By this transformation the additional term $bA$ appears in the $A$-equation. Since it holds $1 + b > 0$, the $ \smash{(A,\widetilde{B}) }$-system can be brought back to the normal form~\eqref{eq:modGL} with transformed coefficients $D > 0$ and $\gamma \in \R$ upon rescaling $A, \smash{\widetilde{B}}$, $ t $ and $ x $. Therefore, to understand the dynamics of bounded perturbations of the periodic solutions~\eqref{eq:solsmodGL2}, it suffices to consider~\eqref{eq:solsmodGL}.
}\end{remark}

\begin{remark}{\rm 
We note that the nonlinear stability of periodic waves against localized perturbations in case of multiple critical modes has been obtained in systems of viscous conservation laws~\cite{JNRZV,JONZVD12,JONZV}. In these works, it is exploited that the conservation law structure yields an additional derivative in front of the nonlinearity, which leads to additional decay. Moreover, it is crucial that different critical modes exhibit different group velocities, so that the interaction between critical modes is harmless. Although the case of a nonlocalized initial phase off-set has been dealt with in~\cite{JONZV}, nonlinear stability of periodic waves against $C_{\mathrm{ub}}^m$-perturbations in such models is still open, see also Remark~\ref{rem:new}.}
\end{remark}

\subsection{Outline of paper}

We start by illustrating the main ideas of our newly developed $L^\infty$-scheme in~\S\ref{sec:ill}. We employ the scheme in~\S\ref{sec:modGL} to study the long-term dynamics of $C_{\mathrm{ub}}^m$-perturbations of the periodic roll solutions~\eqref{eq:solsmodGL} in the modified Ginzburg-Landau system~\eqref{eq:modGL}. The established results are stated in~\S\ref{sec:mainresults} and imply associated results for the real Ginzburg-Landau equation~\eqref{GL}, which can be found in~\S\ref{sec:GL}. Finally, Appendix~\ref{app:aux} is dedicated to auxiliary results to control convolution operators, which arise in the decomposition of the semigroup generated by the linearization about the periodic rolls waves.

\medskip

{\bf Notation.} Throughout, the notation $A \lesssim B$ means that there exists a constant $C > 0$, independent of $A$ and $B$, such that $A \leq CB$.

\medskip

\noindent {\bf Acknowledgements.}  
Funded by the Deutsche Forschungsgemeinschaft (DFG, German Research Foundation) -- Project-ID 491897824.

\section{Illustration of main ideas} \label{sec:ill}

In this section we convey the main ideas behind our $L^\infty$-scheme by considering the scalar nonlinear heat equation
\begin{align}
\partial_t u = \partial_x^2 u + \alpha_1 \left(\partial_x u\right)^{q_1} + \alpha_2 \partial_x \left(u^{q_2}\right), \label{eq:heat1} 
\end{align}
for $q_{1,2} \in \mathbb{N}_{\geq 3}$ and $\alpha_{1,2} \in \R$. The motivation for looking at equation~\eqref{eq:heat1} is that it represents the behavior of the critical modes of the perturbation equations arising in the nonlinear stability analysis of periodic rolls in the real Ginzburg-Landau equation~\eqref{GL} and the modified Ginzburg-Landau system~\eqref{eq:modGL}, see~\eqref{eq:GL} and~\eqref{eq:modGL2}, respectively. That is, these critical modes are only diffusively damped as opposed to exponentially damped. Moreover, they appear in the nonlinearity as a derivative only (translational mode) or  with a derivative in front (critical mode coming from the conservation law in case of the modified Ginzburg-Landau system).

Upon integrating we obtain the Duhamel formula
\begin{align}
u(t) &= \re^{\partial_x^2 t} u_0 + \alpha_1 \int_0^t \re^{\partial_x^2 (t-s)} \left(\partial_x u(s)\right)^{q_1} \de s +  \alpha_2 \int_0^t \re^{\partial_x^2 (t-s)} \partial_x \left(u(s)^{q_2}\right) \de s, \label{eq:du_heat1} 
\end{align}
for the mild solution of~\eqref{eq:heat1} with initial condition $u(0) = u_0$. 

\subsection{Linear estimates}

The smoothing action of the heat semigroup leads to the following well-known decay estimate
\begin{align} \norm[L^\infty]{\partial_x^j \re^{\partial_x^2 t}f} \lesssim t^{-\frac{j}{2}}\norm[L^\infty]{f}, \label{eq:linheat}\end{align}
for $j \in \mathbb{N}_0$ and $f \in L^\infty(\R)$.
Note that without any derivatives, i.e.~for $j = 0$, we obtain no decay from the semigroup. As the estimate in~\eqref{eq:linheat} does not require any form of localization, we refer to it as a pure $L^\infty$-estimate. This contrasts with the estimate
\begin{align} \norm[L^\infty]{\re^{\partial_x^2 t}f} \lesssim t^{-\frac{1}{2p}}\norm[L^p]{f}, \label{eq:linheat2}\end{align}
for $f \in L^p(\R)$ and $p \in [1,\infty)$, where one gives up localization to pick up algebraic decay from the semigroup.

\subsection{Nonlinear estimates}

Take $u_0 \in C_{\mathrm{ub}}^1(\R)$. Then, by standard local existence theory for semilinear parabolic equations, see for instance~\cite{LUN}, there exist $T \in (0,\infty]$ and a unique maximal mild solution $u \in C\big([0,T),C_{\mathrm{ub}}^1(\R)\big)$ satisfying~\eqref{eq:du_heat1}. If $T < \infty$, then it holds
\begin{align} \sup_{t \uparrow T} \norm[W^{1,\infty}]{u(t)} = \infty. \label{eq:blowup}\end{align}

First, we consider the case $\alpha_2 = 0$, so that the toy problem~\eqref{eq:heat1} represents the behavior of the translational mode in the perturbation equations~\eqref{eq:modGL2} and~\eqref{eq:GL}. Taking $t \in (0,T)$ and assuming that the mild solution of~\eqref{eq:heat1} satisfies
\begin{align*} \norm[L^\infty]{\partial_x^j u(s)} \lesssim \frac{1}{(1+s)^{\frac{j}{2}}},\end{align*}
for $s \in [0,t)$, we apply~\eqref{eq:linheat} to estimate the nonlinear term in~\eqref{eq:du_heat1} by
\begin{align} \label{eq:nlheat}
\norm[L^\infty]{\alpha_1 \int_0^t \partial_x^j \re^{\partial_x^2 (t-s)} \left(\partial_x u(s)\right)^{q_1} \de s} \lesssim \int_0^t \frac{1}{(t-s)^{\frac{j}{2}}(1+s)^{\frac{3}{2}}} \de s \lesssim \frac{1}{(1+t)^{\frac{j}{2}}},
\end{align}
for $j = 0,1$, where we use that $q_1 \geq 3$. This indicates that a nonlinear iteration with the template function 
\begin{align*}\eta(t) = \sup_{0 \leq s \leq t} \left[\norm[L^\infty]{u(s)} + \sqrt{1+s} \, \norm[L^\infty]{\partial_x u(s)}\right],\end{align*}
can be closed. Indeed, applying~\eqref{eq:linheat} and~\eqref{eq:nlheat} to~\eqref{eq:du_heat1} one finds a constant $C> 0$ such that
\begin{align} \eta(t) \leq C\left(\|u_0\|_{W^{1,\infty}} + \eta(t)^2\right),\label{eq:heat_key}\end{align}
from which it can be concluded by continuity that, if $u_0 \in C_{\mathrm{ub}}^1(\R)$ satisfies $\|u_0\|_{W^{1,\infty}} < (4C^2)^{-1}$, then~\eqref{eq:blowup} cannot hold and we have $\eta(t) \leq 2C\|u_0\|_{W^{1,\infty}}$ and $T = \infty$. So, we have established global existence and decay of solutions of~\eqref{eq:heat1} with small initial conditions in $C_{\mathrm{ub}}^1(\R)$. We stress that the above nonlinear argument only employs $L^\infty$-estimates. Indeed, both the linear and nonlinear term in~\eqref{eq:du_heat1} are estimated with the aid of~\eqref{eq:linheat}. 

Next, we consider the case $\alpha_1 = 0$, so that the toy problem~\eqref{eq:heat1} represents the behavior of the critical mode originating from the conservation law in the perturbation equation~\eqref{eq:modGL2}. One readily observes that a similar argument as above does not yield enough decay to estimate the nonlinear term in~\eqref{eq:du_heat1}. Indeed, one obtains a bound of the form
$$\int_0^t \frac{1}{(t-s)^{\frac{j}{2}}\sqrt{1+s}} \de s,$$
which cannot be bounded by $(1+t)^{-\frac{j}{2}}$ for $j = 0,1$. Ideally one would take stronger weights in the template function. However, the weights in $\eta(t)$ already reflect the strongest linear decay that can be expected for $C_{\mathrm{ub}}^m$-perturbations, cf.~\eqref{eq:linheat}. Thus, we assume in addition that it holds $u_0 \in L^p(\R)$ for some $p \in [1,2]$ and work with the template function
\begin{align*}\eta(t) = \sup_{0 \leq s \leq t} \left[(1+s)^{\frac{1}{2p}}\norm[L^\infty]{u(s)} + (1+s)^{\frac{1}{2} + \frac{1}{2p}}\norm[L^\infty]{\partial_x u(s)}\right].\end{align*}
Taking $t \in (0,T)$ and assuming that the mild solution to~\eqref{eq:heat1} satisfies
\begin{align*} \norm[L^\infty]{\partial_x^j u(s)} \lesssim \frac{1}{(1+s)^{\frac{j}{2} + \frac{1}{2p}}},\end{align*}
for $s \in [0,t)$ and $j = 0,1$, we apply~\eqref{eq:linheat} to estimate the nonlinear term in~\eqref{eq:du_heat1} by
\begin{align} \label{eq:nlheat2}
\begin{split}
\norm[L^\infty]{\alpha_2 \int_0^t \re^{\partial_x^2 (t-s)} \partial_x\left(u(s)^{q_2}\right) \de s} &\lesssim \int_0^t \frac{1}{\sqrt{t-s}(1+s)^{\frac{3}{2p}}} \de s \lesssim \frac{1}{(1+t)^{\frac{1}{2p}}},\\
\norm[L^\infty]{\alpha_2 \int_0^t \partial_x \re^{\partial_x (t-s)} \partial_x\left(u(s)^{q_2}\right) \de s} &\lesssim \int_0^{\frac{\xi(t) t}{2}} \frac{1}{(t-s) (1+s)^{\frac{3}{2p}}} \de s + \int_{\frac{\xi(t) t}{2}}^t \frac{1}{\sqrt{t-s} (1+s)^{\frac{1}{2} + \frac{3}{2p}}} \de s\\ &\lesssim \frac{1}{(1+t)^{\frac{1}{2} + \frac{1}{2p}}},
\end{split}
\end{align}
where $\xi$ is the characteristic function satisfying $\xi(t) = 0$ for $t \in [0,1]$ and $\xi(t) = 1$ for $t > 1$ and we use $q_2 \geq 3$. Hence, applying~\eqref{eq:linheat2} and~\eqref{eq:nlheat2} to~\eqref{eq:du_heat1} we obtain the key inequality~\eqref{eq:heat_key}. So, we establish global existence and decay of solutions of~\eqref{eq:heat1} with small initial data in $C_{\mathrm{ub}}^1(\R) \cap L^p(\R)$ for some $p \in [1,2]$. We emphasize that localization-induced decay, i.e.~estimate~\eqref{eq:linheat2}, is only used to bound the linear term in~\eqref{eq:du_heat1}, whereas the nonlinear estimate~\eqref{eq:nlheat2} relies on the $L^\infty$-bound~\eqref{eq:linheat} only. Hence, the presented scheme is still a pure $L^\infty$-scheme on the nonlinear level, which is also the case in the upcoming proofs of Theorems~\ref{thm:stabmodGL} and~\ref{thm:partloc}. 

\begin{remark}{\rm
We note that the comparison principle can also be used to conclude that solutions with bounded initial data in the scalar equation~\eqref{eq:heat1} stay bounded. However, the $L^\infty$-scheme illustrated above will be applied to systems of multiple components, i.e.~the perturbation equations~\eqref{eq:modGL2} and~\eqref{eq:GL}, for which a comparison principle is not available. 
}\end{remark}

\section{Analysis in the modified Ginzburg-Landau system} \label{sec:modGL}

We study the dynamics of perturbations of the periodic roll solutions~\eqref{eq:solsmodGL} of the modified Ginzburg-Landau system~\eqref{eq:modGL}. We exploit the gauge invariance of~\eqref{eq:modGL} and write the perturbed solution in polar form
\begin{align}
	(A,B)(t,x) = \left(\sqrt{1-q^2} \,\re^{\ri q x + r(x,t)+ \ri \phi(x,t)},B(t,x)\right). \label{eq:pertmodGL}
\end{align}
Inserting the ansatz~\eqref{eq:pertmodGL} into~\eqref{eq:modGL}, we find that the perturbation $(r,\phi,B)$ satisfies the real autonomous system
\begin{align}
	\label{eq:modGL2}
\begin{split}
	\partial_t r &= \partial_x^2 r + B - 2q \partial_x \phi - \left(1-q^2\right) \left(\re^{2r} - 1\right) + (\partial_x r)^2 - (\partial_x \phi)^2, \\
	\partial_t \phi &= \partial_x^2 \phi + 2q \partial_x r + 2(\partial_x r)(\partial_x\phi), \\
	\partial_t B &= D \partial_x^2 B + \gamma \left(1-q^2\right) \partial_x^2 \left(\re^{2r}\right).
\end{split}
\end{align}
Thus, in order to understand the dynamics of bounded perturbations of the periodic solution~\eqref{eq:solsmodGL}, we study the dynamics of solutions with small, bounded initial data in~\eqref{eq:modGL2}.

We proceed as follows. First, we state the main results in~\S\ref{sec:mainresults}. Then, we collect the necessary local existence theory in~\S\ref{sec:local}. Next, we formulate the perturbation equations with respect to the local wavenumber in~\S\ref{sec:localw}. Subsequently, we study the spectrum of the linearization in~\S\ref{sec:spec} and derive estimates on the associated semigroup in~\ref{sec:semdecomp}. Finally, the proof of the main results can be found in~\S\ref{sec:explong},~\S\ref{sec:partloc} and~\S\ref{sec:q0}. 

\subsection{Main results} \label{sec:mainresults}

Our first result concerns $C_{\mathrm{ub}}^m$-perturbations, i.e.~we impose no localization conditions. We establish that the corresponding solutions to~\eqref{eq:modGL2} exist and stay small in $L^\infty$-norm for exponentially long times. Moreover, we find that the derivative of the perturbation decays diffusively with rate $t^{-1/2}$ for exponentially long times.

\begin{theorem}\label{thm:explong}
Let $q,\gamma \in \R$ and $D > 0$ satisfy~\eqref{e:spec}. Then, there exist $M_0, \epsilon_0 > 0$ such that for all $\epsilon \in (0,\epsilon_0)$, $r_0,\phi_0 \in C_{\mathrm{ub}}^{2}(\R)$ and $B_0 \in C_{\mathrm{ub}}^1(\R)$ satisfying
\begin{align*}
\norm[W^{2,\infty}]{r_0} + \norm[W^{2,\infty}]{\phi_0} + \norm[W^{1,\infty}]{B_0} < \epsilon,
\end{align*}
there exists a mild solution
$$(r,\phi,B) \in C\big(\big[0,\re^{\epsilon_0/\epsilon}-2\big],C_{\mathrm{ub}}^2(\R) \times C_{\mathrm{ub}}^2(\R) \times C_{\mathrm{ub}}^1(\R)\big),$$
to~\eqref{eq:modGL2} with initial condition $(r(0),\phi(0),B(0)) = (r_0,\phi_0,B_0)$, which enjoys the estimates
\begin{align*}
\norm[L^\infty]{r(t)} + \norm[L^\infty]{\partial_x \phi(t)} + \norm[L^\infty]{B(t)} &\leq M_0\epsilon, \\
\norm[W^{1,\infty}]{\partial_x r(t)} + \norm[L^\infty]{\partial_x^2 \phi(t)} + \norm[L^\infty]{\partial_x B(t)} &\leq \frac{M_0\epsilon}{\sqrt{1+t}},
\end{align*}
for all $t \in \big[0,\re^{\epsilon_0/\epsilon}-2\big]$. Additionally, the phase $\phi$ can be bounded as
\begin{align*}
\norm[L^\infty]{\phi(t)} \leq M_0\epsilon\sqrt{1+t},
\end{align*}
for all $t \in \big[0,\re^{\epsilon_0/\epsilon}-2\big]$.
\end{theorem}

The proof of Theorem~\ref{thm:explong} can be found in~\S\ref{sec:explong}.

Our next result concerns bounded, partially localized perturbations. More precisely, in addition to requiring that the initial perturbation $(r_0,\phi_0,B_0)$ and sufficiently many of its derivatives are uniformly continuous and small in $L^\infty$-norm, we demand that $B_0$ is also small in $L^p$-norm for some $p \in [1,\infty)$.

\begin{theorem}\label{thm:partloc}
	Let $p \in [1,\infty)$ and let $q, \gamma \in \R$ and $D > 0$ satisfy~\eqref{e:spec}.
	Then, there exist $M_0, \varepsilon_0 > 0$ such that for all $\varepsilon \in (0,\varepsilon_0)$, $r_0, \phi_0 \in C^2_{\mathrm{ub}}(\R)$ and $B_0 \in L^p(\R) \cap C^1_{\mathrm{ub}}(\R)$ satisfying
	\begin{align*}
		\norm[W^{2,\infty}]{r_0} + \norm[W^{2,\infty}]{\phi_0} + \norm[L^p \cap W^{1,\infty}]{B_0} < \varepsilon,
	\end{align*}
	there exists a global mild solution
	\begin{align*}
		(r,\phi,B) \in C\big([0,\infty), C^2_{\mathrm{ub}}(\R) \times C^2_{\mathrm{ub}}(\R) \times C^1_{\mathrm{ub}}(\R)\big),
	\end{align*}
	to~\eqref{eq:modGL2} with initial condition $(r(0), \phi(0),B(0)) = (r_0,\phi_0,B_0)$ enjoying the estimates
	\begin{align*}
		\norm[L^\infty]{r(t)} + \norm[L^\infty]{\partial_x \phi(t)} + \norm[L^\infty]{B(t)} &\leq M_0 \varepsilon (1+t)^{-\frac{1}{2p}}, \\
		\norm[W^{1,\infty}]{\partial_x r(t)} + \norm[L^\infty]{\partial_x^2 \phi(t)} + \norm[L^\infty]{\partial_x B(t)} &\leq M_0 \varepsilon (1+t)^{-\frac{1}{2}-\frac{1}{2p}},
	\end{align*}
	for all $t \geq 0$.
	Additionally, the phase $\phi$ can be bounded as
	\begin{align*}
		\norm[L^\infty]{\phi(t)} \leq M_0 \varepsilon (1+t)^{\frac{1}{2}-\frac{1}{2p}},
	\end{align*}
	for all $t \geq 0$.
\end{theorem}

The proof of Theorem~\ref{thm:partloc} can be found in~\S\ref{sec:partloc}.

\begin{remark}{ \rm
It turns out that the decay rates obtained in Theorem~\ref{thm:partloc} mesh very well with the result in the nonlocalized setting. In fact, formally taking the limit $p \rightarrow \infty$ in Theorem~\ref{thm:partloc}, we recover the decay rates provided in Theorem~\ref{thm:explong}. However, we are no longer able to establish global existence of mild solutions due to the fact that for $p = \infty$ a logarithm appears in the relevant estimates, cf.~\eqref{e:etaest} and~\eqref{e:etaestLoc}, which explains the exponential long time scale in Theorem~
\ref{thm:explong}. We emphasize that such a loss of global existence is not necessarily artificial. In fact, for $p > 3$ sufficiently small nonnegative solutions in $L^1(\R) \cap L^\infty(\R)$ of the nonlinear heat equation $\partial_t u = \partial_x^2 u + u^p$ exist globally and decay diffusively with rate $t^{-\frac{1}{2}}$, cf.~\cite[Chapter 14]{SU17book}, whereas for $p = 3$ such solutions exist and decay with rate $t^{-\frac{1}{2}}$ only for exponentially long times~\cite{SU03}. Global existence cannot be established for $p = 3$, since all nonnegative nontrivial initial data in $\partial_t u = \partial_x^2 u + u^3$ blow up in finite time~\cite{HAYA}.}
\end{remark}

Finally, we consider the special case $q = 0$, for which the semigroup generated by the linearization of~\eqref{eq:modGL2} exhibits better decay rates when acting on the critical nonlinearity, see~\eqref{eq:refined2} in the upcoming Lemma~\ref{lem:difest}. Due to the improved decay it is possible to extend Theorem~\ref{thm:explong} to a global result, which yields nonlinear stability of the steady state solution $(1,0)$, i.e.~the periodic roll solution~\eqref{eq:solsmodGL} at $q = 0$, of the modified Ginzburg-Landau system~\eqref{eq:modGL} against $C_{\mathrm{ub}}^m$-perturbations. 

\begin{theorem}\label{thm:global}
Let $q = 0$, $\gamma \in \R$ and $D > 0$ with $D + \gamma > 0$. Let $\alpha \in (0,\frac{1}{4})$. Then, there exist $M_0, \epsilon_0 > 0$ such that for all $\epsilon \in (0,\epsilon_0)$, $r_0,\phi_0 \in C_{\mathrm{ub}}^{2}(\R)$ and $B_0 \in C_{\mathrm{ub}}^1(\R)$ satisfying
\begin{align*}
\norm[W^{2,\infty}]{r_0} + \norm[W^{2,\infty}]{\phi_0} + \norm[W^{1,\infty}]{B_0} < \epsilon,
\end{align*}
there exists a global mild solution
$$(r,\phi,B) \in C\left(\left[0,\infty\right),C_{\mathrm{ub}}^2(\R) \times C_{\mathrm{ub}}^2(\R) \times C_{\mathrm{ub}}^1(\R)\right),$$
to~\eqref{eq:modGL2} with initial condition $(r(0),\phi(0),B(0)) = (r_0,\phi_0,B_0)$ enjoying the estimates
\begin{align*}
\norm[L^\infty]{r(t)} + \norm[L^\infty]{B(t)} &\leq M_0\epsilon, \qquad
\norm[W^{1,\infty}]{\partial_x r(t)} + \norm[L^\infty]{\partial_x B(t)} \leq \frac{M_0\epsilon}{\sqrt{1+t}},
\end{align*}
for all $t \geq 0$. Additionally, the phase $\phi$ can be bounded as
\begin{align*}
\norm[L^\infty]{\phi(t)} \leq M_0\epsilon\left(1+t\right)^\alpha, \qquad \norm[W^{1,\infty}]{\partial_x \phi(t)} \leq M_0\epsilon\left(1+t\right)^{-\frac{1}{2} + \alpha},
\end{align*}
for all $t \geq 0$.
\end{theorem}

The proof of Theorem~\ref{thm:global} can be found in~\S\ref{sec:q0}. 

\begin{remark}{\rm
The solution $(r,\phi,B)$ of~\eqref{eq:modGL2}, established in Theorems~\ref{thm:explong},~\ref{thm:partloc} and~\ref{thm:global}, is a classical solution if it holds in addition $r_0,\phi_0 \in C_{\mathrm{ub}}^{3}(\R)$ and $B_0 \in C_{\mathrm{ub}}^2(\R)$. That is, the established solution $(r,\phi,B)$ lies in the space
$$C^1\big(\big[0,\re^{\epsilon_0/\epsilon}-2\big],C_{\mathrm{ub}}^1(\R) \times C_{\mathrm{ub}}^1(\R) \times C_{\mathrm{ub}}(\R)\big) \cap C\big(\big[0,\re^{\epsilon_0/\epsilon}-2\big],C_{\mathrm{ub}}^3(\R) \times C_{\mathrm{ub}}^3(\R) \times C_{\mathrm{ub}}^2(\R)\big).$$
We refer to the upcoming~\S\ref{sec:local} for more details.
}\end{remark}

\begin{remark} \label{rem:optimality}
{ \rm
The conservation law structure and the translational invariance of~\eqref{eq:modGL} induce a 2-parameter family of constant solutions of~\eqref{eq:modGL2} given by
\begin{align} \left(\ln\left(\frac{1-q^2+b}{1-q^2}\right),\tau,b\right), \end{align}
where $\tau,b \in \R$ are such that $1-q^2+b > 0$. Hence, $L^\infty$-decay cannot be expected in any component of solutions of~\eqref{eq:modGL2} with initial conditions in $C_{\mathrm{ub}}^m(\R)$. That is, the decay rates on the $B$- and $r$-components in Theorems~\ref{thm:explong} and~\ref{thm:global} are sharp. The sharpness of the decay at rate $t^{-\frac{1}{2}}$ of their derivatives can be explained, at least on the linear level, by the smoothing action of the associated semigroup. Finally, since $r$ is linearly exponentially damped in~\eqref{eq:modGL2}, it can be formally expressed in terms of $ \phi $ and $ B $, so that in lowest order we obtain
\begin{align} \label{r:exp}
r = - \frac{q}{1-q^2} \partial_x \phi +  \frac{1}{2(1-q^2)} B + \ldots,
\end{align}
This indicates that $\partial_x \phi$ admits the same decay as $B$ and $r$ as long as $q \neq 0$ (and the same for their derivatives). Thus, we expect that the presented decay rates on $\partial_x \phi$ in Theorem~\ref{thm:explong} are sharp for $q \neq 0$. Theorem~\ref{thm:global} shows that for $q = 0$ one can establish stronger decay rates on $\phi$ and its derivatives. We reason that these rates must be optimal in the limit $\alpha \downarrow 0$: $\phi$ cannot be decaying due to translational invariance and taking a derivative yields a decay factor $t^{-\frac{1}{2}}$ (at least on the linear level).
}\end{remark}

\subsection{Local existence and uniqueness} \label{sec:local}

Although the perturbation equation~\eqref{eq:modGL2} is quasilinear, it does not exhibit a loss of regularity. This can be seen by introducing the local wavenumber $\psi = \partial_x \phi$ and the derivative $w = \partial_x r$ of the amplitude, which turns~\eqref{eq:modGL2} into the parabolic semilinear real system
\begin{align}
	\label{eq:modGL3}
\begin{split}
	\partial_t r &= \partial_x^2 r + B - 2q \psi - \left(1-q^2\right) \left(\re^{2r} - 1\right) + w^2 - \psi^2, \\
	\partial_t w &= \partial_x^2 w + \partial_x B - 2q \partial_x \psi - \left(1-q^2\right) \partial_x \left(\re^{2r} - 1\right) + \partial_x \left(w^2\right) - \partial_x\left(\psi^2\right), \\
	\partial_t \phi &= \partial_x^2 \phi + 2q w + 2w\psi, \\
	\partial_t \psi &= \partial_x^2 \psi + 2q \partial_x w + 2\partial_x (w\psi), \\
	\partial_t B &= D \partial_x^2 B + 2\gamma \left(1-q^2\right) \partial_x \left(w \re^{2r}\right).
\end{split}
\end{align}

Local existence and uniqueness for parabolic semilinear systems of the form~\eqref{eq:modGL3} in spaces of bounded and uniformly continuous functions is standard. In particular,~\eqref{eq:modGL3} is of the form
\begin{align}
\label{eq:modGL4}
\partial_t W = \check{L}W + \check{N}(W),
\end{align}
where the linearity $\check{L}$ is a sectorial operator on $X = C_{\mathrm{ub}}(\R,\R^5)$ with dense domain $D(\check{L}) = C_{\mathrm{ub}}^2(\R,\R^5)$, cf.~\cite[Corollary~3.1.9]{LUN}, and the nonlinearity $\check{N}$ is locally Lipschitz continuous from the intermediate space $C_{\mathrm{ub}}^1(\R,\R^5)$ into $X$. Hence, we arrive at the following result, cf.~\cite[Theorem~7.1.3 and Propositions~7.1.8 and~7.1.10]{LUN}.

\begin{proposition} \label{prop:local}
Let $W_0 \in C_{\mathrm{ub}}^1(\R,\R^5)$. Then, there exist a maximal time $T \in (0,\infty]$ and a unique mild solution $W \in C\left([0,T),C_{\mathrm{ub}}^1(\R,\R^5)\right)$ of~\eqref{eq:modGL4} with initial condition $W(0) = W_0$. If we have $T < \infty$, then it holds
\begin{align} \sup_{t \uparrow T} \norm[W^{1,\infty}]{W(t)} = \infty.\end{align}
Moreover, if $W_0 \in C_{\mathrm{ub}}^2(\R,\R^5)$, then $W$ lies in $C\left([0,T),C_{\mathrm{ub}}^2(\R,\R^5)\right) \cap C^1\left([0,T),C_{\mathrm{ub}}(\R,\R^5)\right)$ and is a classical solution of~\eqref{eq:modGL4}.
\end{proposition}

\subsection{Replacing the phase variable by the local wavenumber} \label{sec:localw}

Since the perturbation equation~\eqref{eq:modGL2}, as well as the extended system~\eqref{eq:modGL3}, only depend on derivatives of the phase $\phi$, the dynamics is captured by the $\psi$-, $r$- and $B$-equation in~\eqref{eq:modGL3}. Indeed, the $w$- and $\phi$-components can be recovered by differentiation of $r$ and integration of $\psi$, respectively. Thus, instead of the full system~\eqref{eq:modGL3}, it suffices to consider its $\psi$-, $r$- and $B$-components only, which constitute the system
\begin{align}
	\partial_t V = L V + N_1(V) + \partial_x N_2(V), \qquad V = (r,\psi,B) \in \R^3,
	\label{eq:pertModGLAbbreviation}
\end{align}
with linearity
\begin{align*}
	L = \begin{pmatrix}
		\partial_x^2 - 2 \left(1-q^2\right) & -2q & 1 \\
		2q \partial_x^2 & \partial_x^2 & 0 \\
		2\gamma\left(1-q^2\right)\partial_x^2 & 0 & D\partial_x^2
	\end{pmatrix}
\end{align*}
and nonlinearities
\begin{align*}
	N_1(V) = \begin{pmatrix}
		(\partial_x r)^2 - \psi^2 - \left(1-q^2\right)\left(\re^{2r}-2r-1\right) \\
		0 \\
		0
	\end{pmatrix}, \qquad N_2(V) = \begin{pmatrix}
		0 \\ 2\psi \partial_x r \\ \gamma\left(1-q^2\right) \partial_x \left(\re^{2r}-2r\right)
	\end{pmatrix}.
\end{align*}
We observe that~\eqref{eq:pertModGLAbbreviation} possesses no linear terms with first-order spatial derivatives. Indeed, the operator $L$ contains second derivatives only, so that the Fourier symbol
\begin{align*} \widehat{L}(k) = \begin{pmatrix}
		-k^2 - 2 \left(1-q^2\right) & -2q & 1 \\
		-2qk^2 & -k^2 & 0 \\
		-2\gamma\left(1-q^2\right)k^2 & 0 & -Dk^2
	\end{pmatrix},\end{align*}
of $L$ is analytic in $k^2$. At the critical Fourier mode $k = 0$ we find that $\widehat{L}(0)$ has a semisimple eigenvalue $0$ of algebraic and geometric multiplicity $2$ and a simple negative eigenvalue $-2(1-q^2)$. This suggests a decomposition in diffusive and exponentially damped modes. We will introduce mode filters that facilitate the associated decomposition of the semigroup $\re^{tL}$ generated by $L$ and derive estimates on the components. Such estimates require control on the spectrum of $\widehat{L}(k)$, also for $k$ away from $0$, which we will acquire first in the upcoming subsection.

\subsection{Spectral stability} \label{sec:spec}

In this subsection we establish spectral stability for the periodic roll solution~\eqref{eq:solsmodGL} of the modified Ginzburg-Landau system~\eqref{eq:modGL}. That is, we prove that the spectrum of the linearization $L$ is confined to the open-left half plane, except for a parabolic touching at the origin. Being a constant-coefficient operator, the spectrum of $L$ on $C_{\mathrm{ub}}(\R,\C^3)$ is the same as its spectrum on $L^2(\R,\C^3)$. So, it is determined by its Fourier symbol through the relation
$$\sigma(L) = \bigcup_{k \in \R} \sigma\left(\widehat{L}(k)\right).$$

First, we analyze the spectrum of $\widehat{L}(k)$ away from the critical Fourier mode $k = 0$. We show that the spectrum of $\widehat{L}(k)$ is confined to the open left-half plane for all $k \in \R \setminus \{0\}$. In addition, by studying the spectrum of the Fourier symbol $\widehat{L}(k)$ in the limit $k \to \pm \infty$, we confirm that the sectorial operator $L$ is genuinely of second-order, i.e.~its spectrum is contained in a left-opening parabola, which is needed for obtaining high-frequency semigroup estimates in~\S\ref{sec:semdecomp}.

\begin{lemma}[High-frequency spectrum] \label{lem:spec_cont}
Let $q,\gamma \in \R$ and $D > 0$ satisfy~\eqref{e:spec}. Then, for each $k \in \R \setminus \{0\}$ it holds
$$\sup \Re \sigma(\widehat{L}(k)) < 0.$$
Moreover, $L_\infty = \displaystyle \lim_{\ell \to 0} \ell^2 \widehat{L}\left(\ell^{-1}\right)$ exists and satisfies $\sup \Re \sigma(L_\infty) < 0$.
\end{lemma}
\begin{proof}
Let $k \in \R \setminus \{0\}$. The characteristic polynomial of $\widehat{L}(k)$ is given by
\begin{align*} p(\nu;k) = \nu^3 + a_2(k) \nu^2 + a_1(k)\nu + a_0(k),\end{align*}
with coefficients
\begin{align*}
a_2(k) &= 2\left(1-q^2\right) + (2 + D) k^2, \\
a_1(k) &= \left(2 + k^2 - 6 q^2 + 2 D \left(1 + k^2 - q^2\right) + 2 \gamma (1 - q^2)\right)k^2, \\
a_0(k) &= \left(D k^2 + 2 (D + \gamma)\left(1 - q^2\right) - 4Dq^2\right) k^4.
\end{align*}
By the Routh-Hurwitz criterion all roots of $p(\cdot;k)$ reside in the open left-half plane if and only if $a_2(k),a_0(k)$ are positive and $a_{2}(k)a_{1}(k) > a_{0}(k)$. It is readily seen that~\eqref{e:spec} implies that $a_2(k),a_0(k)$ are positive. For the last condition we consider the quartic
\begin{align*}
Q(k) = a_2(k)a_1(k) - a_0(k) = b_4 k^4 + b_2 k^2 + b_0,
\end{align*}
with coefficients
\begin{align*}
b_4 &= (1 + D)^2,\\
b_2 &= 3\left(1 - 3 q^2\right) + 2q^2 + D(D + \gamma)\left(1 - q^2\right) + (D + \gamma)\left(1 - q^2\right) + 3D\left(1 - q^2\right),\\
b_0 &= 2\left(1 - q^2\right)\left(1 - 3q^2 + (D + \gamma)\left(1 - q^2\right)\right).
\end{align*}
The coefficients $b_0,b_2$ and $b_4$ are all strictly positive by~\eqref{e:spec}, where we use $D + \gamma > 2 Dq^2/(1 - q^2) \geq 0$. So, we have $a_{2}(k)a_{1}(k) > a_{0}(k)$, which completes the proof of the first assertion.

Finally, we observe that
$$L_\infty = \lim_{\ell \to 0} \ell^2 \widehat{L}\left(\ell^{-1}\right) = \begin{pmatrix}
		-1 & 0 & 0 \\
		-2q & -1 & 0 \\
		-2\gamma\left(1-q^2\right) & 0 & -D
	\end{pmatrix},$$
has the eigenvalue $-1$ of algebraic multiplicity $2$ and the simple eigenvalue $-D$, which proves the second assertion.
\end{proof}

Next, we study the spectrum of $\widehat{L}(k)$ in a neighborhood of the critical mode $k = 0$. We establish that the two most critical eigenvalues of $\widehat{L}(k)$ touch the origin in a quadratic tangency as $k$ passes through $0$. This corresponds to two critical spectral curves of $L$ being attached to the origin, one arising due to translational invariance of~\eqref{eq:modGL} and one arising due to the conservation law present in~\eqref{eq:modGL}.

\begin{lemma}[Low-frequency spectrum] \label{lem:specl}
Let $q,\gamma \in \R$ and $D > 0$ satisfy~\eqref{e:spec}. Then, there exists $k_0 > 0$ such that $\widehat{L}(k)$ can be block diagonalized as
\begin{align}\widehat{L}(k) = S(k)^{-1}\begin{pmatrix} \Lambda_c(k) & 0_{2 \times 1} \\ 0_{1 \times 2} & \lambda_s(k) \end{pmatrix}S(k), \qquad |k| < k_0, \label{block}\end{align}
where $S \colon (-k_0,k_0) \to \operatorname{GL}_3(\C)$, $\Lambda_c \colon (-k_0,k_0) \to \C^{2 \times 2}$ and $\lambda_s \colon (-k_0,k_0) \to \C$ are analytic in $k^2$ and satisfy the following assertions:
\begin{itemize}
\item[i.] $\Lambda_c(0) = \Lambda_c'(0) = 0$;
\item[ii.] $\sup \Re \sigma(\Lambda_c(k)) < 0$ for $k \in (-k_0,k_0) \setminus \{0\}$;
\item[iii.] $\sup \Re \sigma(\Lambda_c''(0)) < 0$;
\item[iv.] $\Re(\lambda_s(k)) < 0$ for $k \in (-k_0,k_0)$.
\end{itemize}
Finally, the spectral projection $\mathcal{P} \colon (-k_0,k_0) \to \C^{3 \times 3}$ given by
\begin{align}\mathcal{P}(k) = S(k)^{-1}\begin{pmatrix} 0_{2 \times 2} & 0_{2 \times 1} \\ 0_{1 \times 2} & 1\end{pmatrix} S(k),\label{def:specproj}\end{align}
onto the eigenspace associated with the eigenvalue $\lambda_s(k)$ of $\widehat{L}(k)$ is analytic in $k^2$ and satisfies
\begin{align}
\begin{split}
\mathcal{P}(0) &= \begin{pmatrix}
 1 & \frac{q}{1-q^2} & -\frac{1}{2 \left(1-q^2\right)} \\
 0 & 0 & 0 \\
 0 & 0 & 0
\end{pmatrix}, \\ \mathcal{P}''(0) &= \begin{pmatrix} \frac{\gamma}{1-q^2} - \frac{2q^2}{\left(1 - q^2\right)^2} & \frac{2\gamma q}{\left(1 - q^2\right)^2}-\frac{4q^3}{\left(1 - q^2\right)^3} & \frac{1 + 3 q^2}{2\left(1 - q^2\right)^3} - \frac{D + 2 \gamma}{2\left(1 - q^2\right)^2}\\
\frac{2q}{1-q^2} & \frac{2q^2}{\left(1 - q^2\right)^2} & -\frac{q}{\left(1 - q^2\right)^2} \\
2\gamma & \frac{2\gamma q}{1-q^2} &-\frac{\gamma}{1-q^2}\end{pmatrix}.\end{split} \label{e:specproj}\end{align}
\end{lemma}
\begin{proof}
Recall that $\widehat{L}(k)$ is analytic in $k^2$ and that
$$\widehat{L}(0) = \begin{pmatrix}
		- 2 \left(1-q^2\right) & -2q & 1 \\
		0 & 0 & 0 \\
		0 & 0 & 0
\end{pmatrix},$$
has a semisimple eigenvalue $0$ of multiplicity $2$ and a simple eigenvalue $-2(1-q^2)$. Hence, standard analytic perturbation theory~\cite[Chapter~II.1]{KAT} implies that there exists $\mu_0,k_0 > 0$ such that for $|k| \leq k_0$ the smallest eigenvalue $\lambda_s(k)$ of $\widehat{L}(k)$ is simple, analytic in $k^2$, satisfies $\Re(\lambda_s(k)) < -\mu_0$ and can be separated from the rest of the spectrum. Consequently, $\widehat{L}(k)$ can be block diagonalized as in~\eqref{block}, where the associated basis transformation $S(k) \in \operatorname{GL}_3(\C)$ and upper $(2\times 2$)-block $\Lambda_c(k)$ can be chosen analytic in $k^2$, cf.~\cite[Chapter~II.4]{KAT}. Clearly, $\mathcal{P}(k)$ is the spectral projection onto the eigenspace associated with the eigenvalue $\lambda_s(k)$ of $\widehat{L}(k)$ for $|k| < k_0$. Following~\cite[Chapter~II.2]{KAT} we find that $\mathcal{P}(k)$ must be analytic in $k^2$ and satisfies~\eqref{e:specproj}. In particular, $\mathcal{P}''(0)$ can be computed with the aid of formula~(2.14) in~\cite[Chapter~II.2]{KAT}.

All that remains is to verify the assertions i.-iii. The first assertion follows by the facts that $0$ is a semisimple eigenvalue of $\widehat{L}(k)$ of multiplicity $2$ and that $\Lambda_c(k)$ is analytic in $k^2$. Moreover, the second assertion is directly implied by Lemma~\ref{lem:spec_cont}. Verifying the third assertion is more elaborate. We determine how the semisimple eigenvalue $0$ of $\widehat{L}(k)$ at $k = 0$ splits for small $k \in (-k_0,k_0)$. Proceeding as in~\cite[Chapter~II.2.3]{KAT} we first compute the eigenvalues of the matrix $\frac{1}{2} \mathcal{P}_c \widehat{L}''(0) \mathcal{P}_c$, where $\mathcal{P}_c = I-\mathcal{P}(0)$ is the spectral projection on the two-dimensional neutral eigenspace of $\widehat{L}(0)$. Thus, we calculate
$$\widehat{L}''(0) = \begin{pmatrix}
		-2 & 0 & 0 \\
		-4q & -2 & 0 \\
		-4\gamma\left(1-q^2\right) & 0 & -2D
	\end{pmatrix},$$
and find that the eigenvalues of $\frac{1}{2} \mathcal{P}_c \widehat{L}''(0) \mathcal{P}_c$ are $0$ and
\begin{align*}
\lambda^{(1)}_\pm = -\frac{1}{2}\left(1+D+\gamma\right) + \frac{q^2}{1-q^2} \pm \sqrt{\left(-\frac{1}{2}\left(1+D+\gamma\right) + \frac{q^2}{1-q^2}\right)^2 - D - \gamma + \frac{2 D q^2}{1 - q^2}}.
\end{align*}
It follows by~\eqref{e:spec} that $\lambda^{(1)}_\pm$ are both strictly negative, where we use $-\frac{1}{2} + q^2/(1-q^2) < 0$ and $D+\gamma > 2 Dq^2/(1 - q^2) \geq 0$. Following~\cite[Chapter~II.2.3]{KAT} the eigenvalues $\lambda_{c,\pm}(k)$ of $\widehat{L}(k)$, which converge to the semisimple eigenvalue $0$ of $\widehat{L}(0)$ as $k \to 0$, are $C^2$ (as are the associated eigenvectors) and admit the expansion
\begin{align*}
\lambda_{c,\pm}(k) = \lambda_\pm^{(1)} k^2 + \mathcal{O}(k^3), \qquad k \in (-k_0,k_0),
\end{align*}
taking $k_0 > 0$ smaller if necessary. By assertion i.~and the fact that $\lambda_{c,\pm}(k)$ are the eigenvalues of $\Lambda_c(k)$, it follows that $\lambda_{c,\pm}''(0)$ are the eigenvalues of $\Lambda_c''(0)$. Since $\lambda_{c,\pm}''(0) = 2\lambda_\pm^{(1)}$ are strictly negative, the third assertion follows, which completes the proof.
\end{proof}

\begin{remark} \label{rem:spectralStabilitygGL}
{\rm
The spectral stability condition~\eqref{e:spec} used in Lemmas~\ref{lem:spec_cont} and~\ref{lem:specl} can be formally derived. Indeed, inserting the formally obtained expansion~\eqref{r:exp} for $r$ in the equations for $ \phi $ and $ B $ yields
\begin{align*}
\partial_T \phi & =  \left(1- \frac{2q^2}{1-q^2} \right)\partial_x^2 \phi +   \frac{q}{1-q^2} \partial_x  B + \ldots , \\
\partial_T B & =  (D+\gamma)  \partial_x^2 B - 2 \gamma  q  \partial_x^3 \phi  + \ldots.
\end{align*}
By making a Fourier ansatz we obtain the following spectral stability criteria
\begin{equation*}
\left(1- \frac{2q^2}{1-q^2} \right)(D+\gamma)  +  \frac{2 \gamma q^2}{1-q^2}  > 0,  \qquad 1- \frac{2q^2}{1-q^2} > 0,
\end{equation*}
which are equivalent to~\eqref{e:spec}. We note that in case $ \gamma = 0 $ both conditions reduce to the Eckhaus criterion $ q^2 <  1/3 $.
}\end{remark}

\subsection{Semigroup decomposition and associated estimates} \label{sec:semdecomp}

We decompose the semigroup $\re^{tL}$ generated by the sectorial operator $L$ on $C_{\mathrm{ub}}(\R,\R^3)$ in a diffusive part and an exponentially damped part. First, we note that the temporal Green's function associated with $L$ is given by
\begin{align*} \curlG(z,t) = \int_\R \re^{t\widehat{L}(k)} \re^{\ri kz} \de k.\end{align*}
Next, we recall that in a neighborhood of $k = 0$ the Fourier symbol $\widehat{L}(k)$ of $L$ admits a block diagonalization~\eqref{block} by Lemma~\ref{lem:spec_cont}. We introduce associated mode filters
$$P_c(k) = \chi(k) S(k)^{-1} \begin{pmatrix} I_{2 \times 2} & 0_{2 \times 1} \\ 0_{1 \times 2} & 0 \end{pmatrix} S(k),\qquad
P_s(k) = \chi(k) S(k)^{-1} \begin{pmatrix} 0_{2 \times 2} & 0_{2 \times 1} \\ 0_{1 \times 2} & 1 \end{pmatrix} S(k),$$
where $\chi \colon \R \to \R$ is a smooth cut-off function, whose closed support lies inside $(-k_0,k_0)$ and satisfies $\chi(k) = 1$ for $|k| \leq k_0/2$. This then leads to the Green's function decomposition
\begin{align*} \curlG(z,t) = \curlG_c(z,t) + \curlG_e(z,t),\end{align*}
with
\begin{align} \label{eq:greenG}
\begin{split}
\curlG_c(z,t) &= \int_\R \re^{t\widehat{L}(k)} P_c(k) \re^{\ri kz} \de k,\\
\curlG_e(z,t) &= \int_\R \re^{t\widehat{L}(k)} P_s(k) \re^{\ri kz} \de k + \int_\R \re^{t\widehat{L}(k)} (1-\chi(k)) \re^{\ri kz} \de k.
\end{split}
\end{align}
The semigroup $\re^{tL}$ decomposes accordingly
\begin{align}\re^{tL} = S_c(t) + S_e(t),\label{eq:semdecomp}\end{align}
where we denote
\begin{align} \left(S_j(t)f\right)(x) = \int_\R \curlG_j(x-y,t) f(y) \de y, \qquad j = c,e. \label{eq:semg}\end{align}

First, we obtain estimates on the diffusive part $S_c(t)$ of the semigroup $\re^{tL}$.

\begin{lemma}[Diffusive semigroup estimate] \label{lem:difest}
Let $q,\gamma \in \R$ and $D > 0$ satisfy~\eqref{e:spec}. Let $m,n \in \N_0$ and $p \in [1,\infty]$. Then, there exists a constant $C_{m,n}>0$ such that the estimate
\begin{align*}
\norm[L^\infty]{\partial_x^n S_c(t) \partial_x^m f} &\leq C_{m,n} \left(1+t\right)^{-\frac{n+m}{2} - \frac{1}{2p}} \left(\frac{\norm[L^p \cap L^\infty]{f_1}}{1+t} + \norm[L^p \cap L^\infty]{f_2} + \norm[L^p \cap L^\infty]{f_3}\right),
\end{align*}
holds for all $t \geq 0$, $f = (f_1,f_2,f_3) \in C_{\mathrm{ub}}^m(\R,\R^3) \cap L^p(\R,\R^3)$. Moreover, the estimates
\begin{align}
\norm[L^\infty]{S_c(t) \partial_x^m \begin{pmatrix} g \\ 0 \\ 0\end{pmatrix}} &\lesssim (1+t)^{-\frac{m+1}{2}}\norm[L^\infty]{\partial_x g}, \label{eq:refined1} \\
\norm[L^\infty]{S_c(t) \begin{pmatrix} -h \\ 0 \\ \gamma \partial_x^2 h\end{pmatrix}} &\lesssim q\frac{\norm[L^\infty]{h}}{1+t} + \frac{\norm[L^\infty]{h}}{(1+t)^2},\label{eq:refined2}
\end{align}
are satisfied for all $t \geq 0$, $g \in C_{\mathrm{ub}}^{\max\{1,m\}}(\R)$ and $h \in C_{\mathrm{ub}}^2(\R)$.
\end{lemma}
\begin{proof}
Let $t \geq 0$, $f \in C_{\mathrm{ub}}^m(\R,\R^3) \cap L^p(\R,\R^3)$, $g \in C_{\mathrm{ub}}^{\max\{1,m\}}(\R)$ and $h \in C_{\mathrm{ub}}^2(\R)$. We first note that by Lemma~\ref{lem:specl} the spectral projection $\mathcal{P}(k)$, defined in~\eqref{def:specproj}, is analytic in $k^2$ and satisfies
\begin{align} \label{eq:specid} \left(I_{3 \times 3} - \mathcal{P}(0)\right) \begin{pmatrix} 1 \\ 0 \\ \gamma k^2\end{pmatrix} - \frac{k^2}{2}\mathcal{P}''(0) \begin{pmatrix} 1 \\ 0 \\ 0\end{pmatrix} = k^2\begin{pmatrix} \frac{q^2}{\left(1 - q^2\right)^2} \\ -\frac{q}{1 - q^2} \\ 0\end{pmatrix},\end{align}
for any $k \in \R$. Integrating by parts and using~\eqref{eq:specid} at $k = 0$ yields
\begin{align} \label{eq:intids}
\begin{split}
\partial_x^n S_c(t) \partial_x^m f &= \ri^{n+m} \int_\R\int_\R \re^{t\widehat{L}(k)} P_c(k) k^{n+m} \re^{\ri k(x-y)} \de k f(y) \de y\\
&= \ri^{n+m} \int_\R\int_\R \re^{t\widehat{L}(k)} P_c(k) \left(I_{3 \times 3} - \mathcal{P}(k)\right) k^{n+m} \re^{\ri k(x-y)} \de k f(y) \de y\\
&= \ri^{n+m} \int_\R\int_\R \re^{t\widehat{L}(k)} P_c(k) \left(I_{3 \times 3} - \mathcal{P}(0)\right) k^{n+m} \re^{\ri k(x-y)} \de k \begin{pmatrix} 0 \\ f_2(y) \\ f_3(y)\end{pmatrix} \de y\\
&\qquad +\, \ri^{n+m} \int_\R\int_\R \re^{t\widehat{L}(k)} P_c(k) \frac{\mathcal{P}(0) - \mathcal{P}(k)}{k^2} k^{n+m+2} \re^{\ri k(x-y)} \de k f(y) \de y,
\end{split}
\end{align}
The first estimate now follows directly by bounding the latter two integrals with the aid of Lemmas~\ref{lem:semigroupEstimate1}, whose assumptions are satisfied by Lemma~\ref{lem:specl}, where for the first integral we take $\varpi(k) = \chi(k)\left(I_{3 \times 3} - \mathcal{P}(0)\right)$ and for the second integral we take the smooth function $\varpi(k) = \chi(k) \frac{\mathcal{P}(0) - \mathcal{P}(k)}{k^2}$.

For the second estimate~\eqref{eq:refined1} we substitute $f = (g,0,0)$ and $n = 0$ in the chain of equalities~\eqref{eq:intids} and integrate by parts once again. We obtain
\begin{align*}
S_c(t) \partial_x^m \begin{pmatrix} g \\ 0 \\ 0\end{pmatrix} &= \ri^{m-1}\int_\R\int_\R \re^{t\widehat{L}(k)} P_c(k) \frac{\mathcal{P}(0) - \mathcal{P}(k)}{k^2} k^{m+1} \re^{\ri k(x-y)} \de k \begin{pmatrix} \partial_y g(y) \\ 0 \\ 0\end{pmatrix} \de y.
\end{align*}
As above we apply Lemma~\ref{lem:semigroupEstimate1} to the right-hand side of the latter identity, which then yields estimate~\eqref{eq:refined1}.

Finally, we consider~\eqref{eq:intids} with $f = (0,0,-\gamma h)$, $n = 0$ and $m = 2$ and with $f = (h,0,0)$ and $n = m = 0$. We then add the resulting identities and use~\eqref{eq:specid} to arrive at
\begin{align*}
S_c(t) \begin{pmatrix} h \\ 0 \\ -\gamma \partial_x^2 h\end{pmatrix} &= \int_\R\int_\R \re^{t\widehat{L}(k)} P_c(k) \begin{pmatrix} \frac{q^2}{\left(1 - q^2\right)^2} \\ -\frac{q}{1 - q^2} \\ 0\end{pmatrix} k^2 \re^{\ri k(x-y)} \de k h(y) \de y\\
&\qquad + \, \int_\R\int_\R \re^{t\widehat{L}(k)} P_c(k) \frac{\mathcal{P}(0) - \mathcal{P}(k)}{k^2} k^4 \re^{\ri k(x-y)} \de k \begin{pmatrix} 0 \\ 0 \\ \gamma h(y)\end{pmatrix} \de y,\\
&\qquad + \, \int_\R\int_\R \re^{t\widehat{L}(k)} P_c(k) \frac{\mathcal{P}(0) + \frac{k^2}{2}\mathcal{P}''(0) - \mathcal{P}(k)}{k^4} k^4 \re^{\ri k(x-y)} \de k \begin{pmatrix} h(y) \\ 0 \\ 0\end{pmatrix} \de y,
\end{align*}
which yields the last estimate~\eqref{eq:refined2} by applying Lemma~\ref{lem:semigroupEstimate1} again.
\end{proof}

Subsequently, we establish exponential decay for the residual part $S_e(t)$ of the semigroup $\re^{tL}$.

\begin{lemma}[Exponential semigroup estimate] \label{lem:expest}
Let $q,\gamma \in \R$ and $D > 0$ satisfy~\eqref{e:spec}. Let $n,m \in \N_0$ with $n + m \leq 1$. Then, there exists $\mu_0 > 0$ such that the estimate
\begin{align}
\norm[L^\infty]{\partial_x^n S_e(t) \partial_x^m f} \lesssim \left(1+t^{-\frac{n+m}{2}}\right) \re^{-\mu_0 t} \norm[L^\infty]{f}, \label{e:semest2}
\end{align}
holds for all $t > 0$ and $f \in C_{\mathrm{ub}}^m(\R,\R^3)$.
\end{lemma}
\begin{proof}
Recall that $\widehat{L}(k)$ is analytic in $k^2$ and satisfies
$$\sup \Re \sigma\left(\widehat{L}(k)\right) < 0, \qquad k \in \R \setminus \{0\},$$
by Lemma~\ref{lem:spec_cont}. Hence, using that the spectrum of $\widehat{L}(k)$ depends continuously on $k$, there exists $\smash{\widetilde{L}} \in C^2\big(\R,\C^{3 \times 3}\big)$ satisfying
\begin{align} \widetilde{L}(k) = \widehat{L}(k), \qquad k \in \R \setminus \left[-\frac{k_0}{2},\frac{k_0}{2}\right], \label{e:modprop1}\end{align}
and
\begin{align} \sup \Re \sigma\left(\widetilde{L}(k)\right) < 0, \qquad k \in \R. \label{e:modprop2}\end{align}

Now let $t > 0$ and $f \in C_{\mathrm{ub}}^m(\R,\R^3)$. Using integration by parts, identity~\eqref{e:modprop1} and the fact that $\chi(k) = 1$ for $k \in [-\frac{k_0}{2},\frac{k_0}{2}]$, we rewrite
\begin{align*} \partial_x^n S_e(t) \partial_x^m f &= \ri^{n+m} \left(\int_\R\int_\R \re^{t\widehat{L}(k)} P_s(k) k^{n+m} \re^{\ri k(x-y)} \de k f(y) \de y + \, \int_\R \int_\R \re^{t\widetilde{L}(k)} k^{n+m} \re^{\ri k(x-y)} \de k f(y) \de y\right.\\
&\qquad \qquad \qquad \left. - \, \int_\R\int_\R \re^{t\widetilde{L}(k)} \chi(k) k^{n+m} \re^{\ri k(x-y)} \de k f(y) \de y\right),
\end{align*}
as a sum of three integrals. To bound the first integral, we observe that by Lemma~\ref{lem:specl} the assumptions of Lemma~\ref{lem:semigroupEstimate1} are satisfied. Therefore, there exists $\mu_1 > 0$ such that this integral can be bounded as $\lesssim \re^{-\mu_1 t}\norm[L^\infty]{f}$. For the second integral we note that the assumptions of Lemma~\ref{lem:semigroupEstimate3} are satisfied, where we use Lemma~\ref{lem:spec_cont}, the fact that $\widehat{L}$ is quadratic in $k$, and the identities~\eqref{e:modprop1} and~\eqref{e:modprop2}. Hence, there exists $\mu_2 > 0$ such that the second integral can be bounded as $\lesssim (1+t^{-\frac{m+n}{2}}) \re^{-\mu_2 t} \norm[L^\infty]{f}$. Finally, to bound the third integral, we apply Lemma~\ref{lem:semigroupEstimate1} again with $j = 0$ and $S(k) = I_{3 \times 3}$. Thus, using~\eqref{e:modprop2}, Lemma~\ref{lem:semigroupEstimate1} yields $\mu_3 > 0$ such that the third integral can be bounded as $\lesssim \re^{-\mu_3 t} \norm[L^\infty]{f}$. This completes the proof.
\end{proof}

\subsection{Proof of Theorem~\ref{thm:explong}} \label{sec:explong}

Having estimated the diffusive and exponentially damped part of the semigroup generated by the linearization $L$ of~\eqref{eq:pertModGLAbbreviation} on $C_{\mathrm{ub}}(\R,\R^3)$, we are now able to prove Theorem~\ref{thm:explong}. 

\begin{proof}[Proof of Theorem~\ref{thm:explong}]
By Proposition~\ref{prop:local} there exist a maximal time $T \in (0,\infty]$ and a unique mild solution $W \in C\left([0,T),C_{\mathrm{ub}}^1(\R,\R^5)\right)$ of~\eqref{eq:modGL4} with initial condition $W(0) = (r_0,\partial_x r_0,\phi_0,\partial_x \phi_0,B_0) \in C_{\mathrm{ub}}^1(\R,\R^5)$. If we have $T < \infty$, then it holds
\begin{align*} \sup_{t \uparrow T} \norm[W^{1,\infty}]{W(t)} = \infty.\end{align*}
Write $W(t) = (r(t),w(t),\phi(t),\psi(t),B(t))$ and define $V \in C\big([0,T),C_{\mathrm{ub}}^2(\R) \times C_{\mathrm{ub}}^1(\R) \times C_{\mathrm{ub}}^1(\R)\big)$ by $V(t) = (r(t),\psi(t),B(t))$. Then, $V(t)$ is a mild solution of~\eqref{eq:pertModGLAbbreviation} with $V(0) = V_0 := (r_0,\partial_x \phi_0,B_0)$. We define the template function $\eta \colon [0,T) \to \R$ by $\eta(t) = \eta_1(t) + \eta_2(t)$, where we denote
\begin{align*} \eta_1(t) &= \sup_{0 \leq s \leq t} \left(\norm[L^\infty]{V(s)} + \sqrt{1+s} \norm[L^\infty]{\partial_x V(s)}\right), \\
\eta_2(t) &= \sup_{0 \leq s \leq t} \left(\sqrt{1+s} \norm[L^\infty]{\partial_x^2 r(s)} + \frac{\norm[L^\infty]{\phi(s)}}{\sqrt{1+s}}\right).\end{align*}
Clearly, $\eta$ is continuous and, if $T < \infty$, then it holds
\begin{align} \lim_{t \uparrow T} \eta(t) = \infty. \label{e:blowupeta}\end{align}

Our goal is to prove that there exists a constant $C > 1$ such that for all $t \in [0,T)$ with $\eta(t) \leq 1$ we have
\begin{align}
\eta(t) \leq C\left(\epsilon + \eta(t)\left(\eta(t) + \epsilon\right) \ln(2+t)\right). \label{e:etaest}
\end{align}
Then, taking
$$\epsilon_0 = \frac{1}{4C^2 + 2C}, \qquad M_0 = 2C,$$
it follows by the continuity, monotonicity and non-negativity of $\eta$ that, provided $\epsilon \in (0,\epsilon_0)$, we have $\eta(t) \leq M_0\epsilon = 2C\epsilon \leq 1$ for all $t \in [0,T) \cap \big[0,\re^{\epsilon_0/\epsilon}-2\big]$. Indeed, given $t \in [0,T) \cap \big[0,\re^{\epsilon_0/\epsilon}-2\big]$ with $\eta(s) \leq 2C\epsilon$ for each $s \in [0,t]$, we arrive at
$$\eta(t) \leq C\left(\epsilon + \left(4C^2 + 2C\right)\epsilon^2\ln(2+t)\right) < 2C\epsilon,$$
by estimate~\eqref{e:etaest} and the fact that $\left(4C^2 + 2C\right)\epsilon\ln(2+t) \leq 1$. Thus, if~\eqref{e:etaest} is satisfied, then we have $\eta(t) \leq 2C\epsilon$, for all $t \in [0,T) \cap \big[0,\re^{\epsilon_0/\epsilon}-2\big]$, which implies by~\eqref{e:blowupeta} that it must hold $T > \re^{\epsilon_0/\epsilon}-2$. Consequently, $\eta(t) \leq M_0\epsilon$ is satisfied for all $t \in \big[0,\re^{\epsilon_0/\epsilon}-2\big]$, which proves the result.

It remains to prove the key estimate~\eqref{e:etaest}. To this end, assume that $t \in [0,T)$ is such that $\eta(t) \leq 1$. We start by bounding $V(t)$ and its derivative. Since $V \in C\big([0,T),C_{\mathrm{ub}}^1(\R,\R^3)\big)$ is a mild solution to~\eqref{eq:pertModGLAbbreviation}, it satisfies the Duhamel formulation
\begin{align}
\partial_x^m V(t) = \re^{t L} \partial_x^m V_0 + \int_0^t \re^{(t-s) L} \partial_x^m \left(N_1(V(s)) + \partial_x N_2(V(s))\right) \de s. \label{duhamelV}
\end{align}
for $m = 0,1$. The nonlinearities obey the estimates
\begin{align}
\begin{split}
\norm[L^\infty]{N_1(V(s))} &\lesssim \norm[W^{1,\infty}]{V(s)}^2 \lesssim \eta_1(t)^2,\\
\norm[L^\infty]{N_2(V(s))} &\lesssim \norm[L^\infty]{V(s)} \norm[L^\infty]{\partial_x V(s)} \lesssim \frac{\eta_1(t)^2}{\sqrt{1+s}},\\
\norm[L^\infty]{\partial_x N_j(V(s))} &\lesssim \norm[W^{1,\infty}]{V(s)}\left(\norm[L^\infty]{\partial_x V(s)} + \norm[L^\infty]{\partial_x^2 r(s)}\right) \lesssim \frac{\eta_1(t)\left(\eta_1(t) + \eta_2(t)\right)}{\sqrt{1+s}},
\end{split} \label{eq:nonl1}
\end{align}
for $j = 1,2$ and all $s \in [0,t]$, where we use $\eta(t) \leq 1$. Thus, recalling the semigroup decomposition~\eqref{eq:semdecomp}, applying Lemmas~\ref{lem:difest} and~\ref{lem:expest} to~\eqref{duhamelV}, using~\eqref{eq:nonl1} and exploiting that the nonlinearity $N_1(V)$ is of the form $(\ast,0,0)^\top$, we obtain
\begin{align*}
\norm[L^\infty]{V(t)} &\lesssim \norm[L^\infty]{V_0} + \int_0^t \frac{\norm[L^\infty]{\partial_x N_1(V(s))}}{\sqrt{1+t-s}} \de s + \int_0^t \frac{\norm[L^\infty]{N_2(V(s))}}{\sqrt{1+t-s}} \de s\\
& \qquad + \, \int_0^t \re^{-\mu_0 (t-s)}\left(\norm[L^\infty]{N_1(V(s))} + \frac{\norm[L^\infty]{N_2(V(s))}}{\sqrt{t-s}}\right) \de s\\
&\lesssim \norm[L^\infty]{V_0} + \int_0^t \frac{\eta_1(t)\left(\eta_1(t) + \eta_2(t)\right)}{\sqrt{1+t-s}\sqrt{1+s}} \de s + \int_0^t \frac{\eta_1(t)^2}{\re^{\mu_0 (t-s)}} \left(1 + \frac{1}{\sqrt{1+s}\sqrt{t-s}}\right) \de s\\
&\lesssim \epsilon + \eta_1(t)\left(\eta_1(t) + \eta_2(t)\right),
\end{align*}
and
\begin{align*}
\norm[L^\infty]{\partial_x V(t)} &\lesssim \frac{\norm[W^{1,\infty}]{V_0}}{\sqrt{1+t}} + \int_0^t \frac{\norm[L^\infty]{\partial_x N_1(V(s))}}{1+t-s} \de s + \int_0^t \frac{\norm[L^\infty]{N_2(V(s))}}{1+t-s} \de s\\
& \qquad + \, \int_0^t \re^{-\mu_0 (t-s)}\left(\norm[L^\infty]{\partial_x N_1(V(s))} + \frac{\norm[L^\infty]{\partial_x N_2(V(s))}}{\sqrt{t-s}}\right) \de s\\
&\lesssim \frac{\norm[W^{1,\infty}]{V_0}}{\sqrt{1+t}} + \int_0^t \frac{\eta_1(t)\left(\eta_1(t) + \eta_2(t)\right)}{(1+t-s)\sqrt{1+s}} \de s + \int_0^t \frac{\eta_1(t)\left(\eta_1(t) + \eta_2(t)\right)}{\re^{\mu_0 (t-s)}\sqrt{1+s}} \left(1 + \frac{1}{\sqrt{t-s}}\right) \de s\\
&\lesssim \frac{\epsilon + \eta_1(t)\left(\eta_1(t) + \eta_2(t)\right)\ln(2+t)}{\sqrt{1+t}}.
\end{align*}
Combining the latter two estimates we arrive at a constant $C_1 > 1$ such that
\begin{align}
\eta_1(t) \leq C_1\left(\epsilon + \eta_1(t)\left(\eta_1(t) + \eta_2(t)\right)\ln(2+t)\right). \label{e:etaest1}
\end{align}

Next we bound $\partial_x^2 r(t)$ and $\phi(t)$, which satisfy the Duhamel formulations
\begin{align}
\begin{split}
\partial_x^2 r(t) &= \re^{t L_0} \partial_x^2 r_0 + \int_0^t \re^{(t-s) L_0} \partial_x \mathcal{N}_1(s) \de s,\\
\phi(t) &= \re^{t \partial_x^2} \phi_0 + \int_0^t \re^{(t-s) \partial_x^2} \mathcal{N}_2(s) \de s,
\end{split} \label{duhamelphir}
\end{align}
where we denote $L_0 = \partial_x^2 - 2(1-q^2)$ and
\begin{align*}
\mathcal{N}_1(s) &= \partial_x B(s) - 2q\partial_x\psi(s) + \partial_x \left(\partial_x r(s)\right)^2 - \partial_x\left(\psi(s)^2\right) - \left(1-q^2\right)\partial_x \left(\re^{2r(s)} - 1 - 2r(s)\right),\\
\mathcal{N}_2(s) &= 2q\partial_x r(s) + 2\psi(s)\left(\partial_x r(s)\right).
\end{align*}
We establish the estimates
\begin{align}
\begin{split}
\norm[L^\infty]{\mathcal{N}_j(s)} &\lesssim \norm[L^\infty]{\partial_x V(s)} \lesssim \frac{\eta_1(t)}{\sqrt{1+s}},\\
\end{split} \label{eq:nonl2}
\end{align}
for $j = 1,2$ and $s \in [0,t]$, where we use $\eta(t) \leq 1$. Thus, applying the standard semigroup estimates
\begin{align} \label{e:estphir}
\norm[L^\infty \to L^\infty]{\re^{z \partial_x^2}\partial_x^m } \lesssim z^{-\frac{m}{2}}, \qquad \norm[L^\infty \to L^\infty]{\re^{z L_0}\partial_x^m} \lesssim z^{-\frac{m}{2}}\re^{-2\left(1-q^2\right)z},
\end{align}
for $m = 0,1$ and $z > 0$, cf.~Lemmas~\ref{lem:semigroupEstimate1} and~\ref{lem:semigroupEstimate3}, and using the nonlinear estimate~\eqref{eq:nonl2}, we bound~\eqref{duhamelphir} as
\begin{align}
\label{e:rest}
\norm[L^\infty]{\partial_x^2 r(t)} \lesssim \re^{-2\left(1-q^2\right) t} \norm[W^{2,\infty}]{r_0} + \int_0^t \frac{\re^{-2\left(1-q^2\right)(t-s)} \eta_1(t)}{\sqrt{t-s}\sqrt{1+s}} \de s \lesssim \frac{\epsilon + \eta_1(t)}{\sqrt{1+t}},
\end{align}
and
\begin{align*}
\norm[L^\infty]{\phi(t)} &\lesssim \norm[L^\infty]{\phi_0} + \int_0^t \frac{\eta_1(t)}{\sqrt{1+s}} \de s \lesssim \left(\epsilon + \eta_1(t)\right)\sqrt{1+t}.
\end{align*}
Combining the latter two estimates yields a constant $C_2 > 1$ such that
\begin{align}
\eta_2(t) \leq C_2\left(\epsilon + \eta_1(t)\right). \label{e:etaest2}
\end{align}

Substituting~\eqref{e:etaest2} into~\eqref{e:etaest1} yields
\begin{align}
\eta_1(t) \leq 2C_1C_2\left(\epsilon + \eta_1(t)\left(\eta_1(t) + \epsilon\right)\ln(2+t)\right), \label{e:etaest3}
\end{align}
Subsequently, we employ~\eqref{e:etaest2} and~\eqref{e:etaest3} to establish
\begin{align*}
\eta(t) &= \eta_1(t) + \eta_2(t) \leq \eta_1(t) + C_2(\epsilon + \eta_1(t)) \leq 2C_2\left(\epsilon + \eta_1(t)\right) \\
&\leq 6C_1C_2^2\left(\epsilon + \eta(t)\left(\eta(t) + \epsilon\right)\ln(2+t)\right),
\end{align*}
which proves the key inequality~\eqref{e:etaest} by taking $C = 6C_1C_2^2$.
\end{proof}

\subsection{Proof of Theorem~\ref{thm:partloc}} \label{sec:partloc}

Requiring the $B$-component to be slightly localized yields additional decay on the linear level, which allows us to extend the statement in Theorem~\ref{thm:explong} for exponentially long times to the global result in Theorem~\ref{thm:partloc}.

\begin{proof}[Proof of Theorem~\ref{thm:partloc}]
	We proceed similar to the proof of Theorem~\ref{thm:explong} and only outline the differences.
	The template function $\eta(t)$ in the proof of Theorem~\ref{thm:explong} is adapted to reflect the improved decay rates. That is, we define $\eta : [0,T) \rightarrow \R$ by $\eta(t) = \eta_1(t) + \eta_2(t)$ with
	\begin{align*}
		\eta_1(t) &:= \sup_{0 \leq s \leq t} \left((1+s)^{\frac{1}{2p}} \norm[L^\infty]{V(s)} + (1+s)^{\frac{1}{2p}+\frac{1}{2}} \norm[L^\infty]{\partial_x V(s)}\right), \\
		\eta_2(t) &:= \sup_{0 \leq s \leq t} \left((1+s)^{\frac{1}{2p}-\frac{1}{2}} \norm[L^\infty]{\phi(s)} + (1+s)^{\frac{1}{2p}+\frac{1}{2}} \norm[L^\infty]{\partial_x^2 r(s)}\right).
	\end{align*}
	Instead of establishing the key inequality~\eqref{e:etaest}, our goal is to prove that there exists a constant $C > 1$ such that for all $t \in [0,T)$ with $\eta(t) \leq 1$ we have
	\begin{align}
		\eta(t) \leq C\left(\epsilon + \eta(t)\left(\eta(t) + \epsilon\right)\right). \label{e:etaestLoc}
	\end{align}
	Then, taking $M_0 = 2C$ and $\epsilon_0 = (4C^2 + 2C)^{-1}$, it follows by the continuity, monotonicity and non-negativity of $\eta$ that, provided $\epsilon \in (0,\epsilon_0)$, we have $\eta(t) \leq M_0\epsilon = 2C\epsilon \leq 1$ for all $t \in [0,T)$, which implies that~\eqref{e:blowupeta} cannot hold and we must have $T = \infty$, yielding the result.
	
We now provide the estimates to establish the key inequality~\eqref{e:etaestLoc}. Therefore, we assume $t \in [0,T)$ is such that $\eta(t) \leq 1$. Again, since $V$ is a mild solution of~\eqref{eq:pertModGLAbbreviation}, it satisfies the Duhamel formulation~\eqref{duhamelV}. To estimate the linear part $\re^{tL} \partial_x^m V_0$ in~\eqref{duhamelV} we use $B_0 \in L^p(\R) \cap C_{\mathrm{ub}}^1(\R)$, exploit the semigroup decomposition~\eqref{eq:semdecomp} and apply Lemmas~\ref{lem:difest} and~\ref{lem:expest} to obtain
	\begin{equation}
		\label{eq:linEstLoc}
		\begin{aligned}
			&\norm[L^\infty]{\re^{tL} \partial_x^m V_0} \lesssim \norm[L^\infty]{S_c(t) \partial_x^m \begin{pmatrix} r_0 \\ \partial_x \phi_0 \\ 0\end{pmatrix}} + \norm[L^\infty]{S_c(t) \partial_x^m  \begin{pmatrix} 0 \\ 0 \\ B_0\end{pmatrix}} + \norm[L^\infty]{S_e(t) \partial_x^m V_0} \\
			&\qquad\lesssim (1+t)^{-\frac{m}{2}}\left(\frac{\norm[L^\infty]{r_0}}{1+t} + \frac{\norm[L^\infty]{\phi_0}}{\sqrt{1+t}} + (1+t)^{-\frac{1}{2p}} \norm[L^p \cap W^{1,\infty}]{B_0}\right) + \re^{-\mu_0 t} \norm[W^{1,\infty}]{V_0}\\
&\qquad\lesssim (1+t)^{-\frac{m}{2}-\frac{1}{2p}} \varepsilon,
		\end{aligned}
	\end{equation}
for $m = 0,1$.
	
	To estimate the nonlinear part in~\eqref{duhamelV} we observe, using~\eqref{eq:nonl1} and the improved decay rates encoded in $\eta(t)$, that the nonlinearities obey the refined estimates
	\begin{equation}
		\label{eq:nonlinEstLoc}
		\begin{aligned}
			\norm[L^\infty]{N_1(V(s))} &\lesssim (1+s)^{-\frac{1}{p}} \eta_1(t)^2, \\
			\norm[L^\infty]{N_2(V(s))} &\lesssim (1+s)^{-\frac{1}{p}-\frac{1}{2}} \eta_1(t)^2, \\
			\norm[L^\infty]{\partial_x N_j(V(s))} &\lesssim (1+s)^{-\frac{1}{p}-\frac{1}{2}} \eta_1(t)(\eta_1(t) + \eta_2(t)),
		\end{aligned}
	\end{equation}
	for $j = 1,2$ and all $s \in [0,t]$, where we use $\eta(t) \leq 1$. Then, exploiting the semigroup decomposition~\eqref{eq:semdecomp}, the fact that $N_1(V(s)) = (\ast,0,0)^\top$ and utilizing the estimates in Lemmas~\ref{lem:difest} and~\ref{lem:expest}, we obtain
	\begin{align*}
		&\norm[L^\infty]{\int_0^t \re^{(t-s)L} (N_1(V(s)) + \partial_x N_2(V(s))) \de s}\\
        &\qquad \lesssim \int_0^t \frac{\norm[L^\infty]{N_1(V(s))}}{1+t-s} \de s + \int_0^t \frac{\norm[L^\infty]{N_2(V(s))}}{\sqrt{1+t-s}} \de s + \int_0^t \frac{\norm[L^\infty]{N_1(V(s))} + \norm[L^\infty]{\partial_x N_2(V(s))}}{\re^{\mu_0 (t-s)}} \de s \\
		&\qquad \lesssim \left(\int_0^t (1+t-s)^{-1}(1+s)^{-\frac{1}{p}} \de s + \int_0^t (1+t-s)^{-\frac{1}{2}} (1+s)^{-\frac{1}{p}-\frac{1}{2}} \de s\right) \eta_1(t)^2 \\
		&\qquad \qquad + \, \left(\int_0^t \re^{-\mu_0 (t-s)} (1+s)^{-\frac{1}{p}} \de s + \int_0^t \re^{-\mu_0 (t-s)} (1+s)^{-\frac{1}{p} - \frac{1}{2}}\right) \eta_1(t)(\eta_1(t) + \eta_2(t)) \de s \\
		&\qquad \lesssim (1+t)^{-\frac{1}{2p}} \eta_1(t)(\eta_1(t) + \eta_2(t)),
	\end{align*}
	and
	\begin{align*}
		&\norm[L^\infty]{\int_0^t \re^{(t-s)L} \partial_x (N_1(V(s)) + \partial_x N_2(V(s))) \de s} \\
		&\qquad \lesssim \int_0^t \frac{\norm[L^\infty]{\partial_x N_1(V(s))} + \norm[L^\infty]{N_2(V(s))}}{1+t-s} \de s + \int_0^t \left(\frac{\norm[L^\infty]{\partial_x N_1(V(s))}}{\re^{\mu_0 (t-s)}} + \frac{\norm[L^\infty]{\partial_x N_2(V(s))}}{\re^{\mu_0 (t-s)}\sqrt{t-s}}\right) \de s \\
		&\qquad \lesssim \left(\int_0^t (1+t-s)^{-1} (1+s)^{-\frac{1}{2}-\frac{1}{p}} \de s\right) \eta_1(t)(\eta_1(t) + \eta_2(t)) \\
		&\qquad\qquad + \left(\int_0^t \re^{-\mu_0 (t-s)} (1+s)^{-\frac{1}{2}-\frac{1}{p}}\left(1+\frac{1}{\sqrt{t-s}}\right) \de s\right) \eta_1(t)(\eta_1(t) + \eta_2(t)) \\
		&\qquad \lesssim (1+t)^{-\frac{1}{2p} - \frac{1}{2}} \eta_1(t)(\eta_1(t) + \eta_2(t)).
	\end{align*}
	Together with the linear estimate~\eqref{eq:linEstLoc}, these estimates then yield
	\begin{align*}
		\norm[L^\infty]{\partial_x^m V(t)} &\lesssim (1+t)^{-\frac{1}{2p} - \frac{m}{2}} (\varepsilon + \eta_1(t)(\eta_1(t) + \eta_2(t))),
	\end{align*}
	for $m = 0,1$, which implies that there exists a constant $C_1 > 1$ such that
	\begin{align}
		\label{eq:etaEstLoc1}
		\eta_1(t) \leq C_1 (\varepsilon + \eta_1(t) (\eta_1(t) + \eta_2(t))).
	\end{align}
	
	It remains to provide estimates on $\partial_x^2 r(t)$ and $\phi(t)$.
	Thus, proceeding as in the proof of Theorem~\ref{thm:explong} and exploiting that $\norm[L^\infty]{\partial_x V(s)} \lesssim \smash{(1+s)^{-\frac{1}{2}-\frac{1}{2p}}} \eta_1(t)$ for $s \in [0,t]$, we find
	\begin{equation}
		\label{eq:restLoc}
		\begin{aligned}
			\norm[L^\infty]{\partial_x^2 r(t)} &\lesssim \re^{-2(1-q^2)t} \norm[W^{2,\infty}]{r_0} + \int_0^t \re^{-2(1-q^2)(t-s)} (t-s)^{-\frac{1}{2}} (1+s)^{-\frac{1}{2}-\frac{1}{2p}} \eta_1(t) \de s \\
			&\lesssim (1+t)^{-\frac{1}{2}-\frac{1}{2p}} (\varepsilon + \eta_1(t)).
		\end{aligned}
	\end{equation}
	On the other hand, applying the estimates $\norm[L^\infty]{r(s)} \lesssim (1+s)^{-\frac{1}{2p}} \eta_1(t)$ and $\norm[L^\infty]{(\partial_x \phi(s))(\partial_x r(s))} \lesssim \smash{(1+s)^{-\frac{1}{p}-\frac{1}{2}}} \eta_1(t)^2$ for $s \in [0,t]$ to the Duhamel formula~\eqref{duhamelphir} and using~\eqref{e:estphir}, we establish
	\begin{equation}
		\label{eq:phiEstLoc}
		\begin{aligned}
			\norm[L^\infty]{\phi(t)} &\lesssim \norm[L^\infty]{\phi_0} + \int_0^t (t-s)^{-\frac{1}{2}} (1+s)^{-\frac{1}{2p}} \eta_1(t) \de s + \int_0^t (1+s)^{-\frac{1}{p}-\frac{1}{2}} \eta_1(t)^2 \de s \\
			&\lesssim (1+t)^{\frac{1}{2}-\frac{1}{2p}} (\varepsilon + \eta_1(t)),
		\end{aligned}
	\end{equation}
    where we recall $\eta_1(t) \leq 1$.
	
	Combining the estimates~\eqref{eq:restLoc} and~\eqref{eq:phiEstLoc} yields a constant $C_2 > 1$ such that
	\begin{align}
		\label{eq:etaEstLoc2}
		\eta_2(t) \leq C_2 (\varepsilon + \eta_1(t)).
	\end{align}
	Finally, analogous to the proof of Theorem~\ref{thm:explong} we combine~\eqref{eq:etaEstLoc1} into~\eqref{eq:etaEstLoc2} to obtain the key inequality~\eqref{e:etaestLoc} with $C = 6C_1 C_2^2$.
\end{proof}

\begin{remark}{\upshape
	We highlight that, although Theorem~\ref{thm:partloc} requires partial localization of the initial perturbation, this is only needed for the estimates of the \emph{linear} part $\re^{tL} V_0$ in~\eqref{duhamelV}. The nonlinear argument in the proof of Theorem~\ref{thm:partloc} employs a pure $L^\infty \rightarrow L^\infty$-scheme.
}\end{remark}

\subsection{Proof of Theorem~\ref{thm:global}} \label{sec:q0}

For $q = 0$ the equation for the phase $\phi$ in~\eqref{eq:modGL2} decouples on the linear level. Moreover, the refined semigroup estimate~\eqref{eq:refined2} exhibits quadratic temporal decay instead of just linear decay. These two facts are sufficient to enhance Theorem~\ref{thm:explong} so that it holds globally in time and yields stronger decay for the phase variable without requiring a partial localization as in Theorem~\ref{thm:partloc}.

\begin{proof}[Proof of Theorem~\ref{thm:global}]
We follow the proofs of Theorems~\ref{thm:explong} and~\ref{thm:partloc} and only outline the differences. First, the template function $\eta(t)$ is adapted to accommodate the better decay rates of the phase function. That is, we define $\eta \colon [0,T) \to \R$ by $\eta(t) = \eta_1(t) + \eta_2(t)$ with
\begin{align*} \eta_1(t) &= \sup_{0 \leq s \leq t} \left(\norm[L^\infty]{V(s)} + \sqrt{1+s} \norm[L^\infty]{\partial_x V(s)} + (1+s)^{-\alpha}\left(\norm[L^\infty]{\phi(s)} + \sqrt{1+s}\norm[W^{1,\infty}]{\partial_x \phi(s)}\right)\right), \\
\eta_2(t) &= \sup_{0 \leq s \leq t} \sqrt{1+s} \norm[L^\infty]{\partial_x^2 r(s)}.\end{align*}
Our goal is to prove that there exists a constant $C > 1$ such that for all $t \in [0,T)$ with $\eta(t) \leq 1$ we have
\begin{align}
\eta(t) \leq C\left(\epsilon + \eta(t)\left(\eta(t) + \epsilon\right)\right), \label{e:etaest0}
\end{align}
which yields the desired result as in the proof of Theorem~\ref{thm:partloc}.

We now establish the estimates which lead to the key inequality~\eqref{e:etaest0}. To this end, assume $t \in [0,T)$ is such that $\eta(t) \leq 1$. We observe that the nonlinearity in system~\eqref{eq:pertModGLAbbreviation} can be written as
\begin{align} N_1(V) + \partial_x N_2(V) = \begin{pmatrix} \widetilde{N}_1(r,\phi) - \widetilde{N}_2(r) \\ \partial_x \widetilde{N}_3(r,\phi) \\ \gamma \partial_x^2 \widetilde{N}_2(r)\end{pmatrix}, \label{e:nonldecomp}\end{align}
with
\begin{align*}
\widetilde{N}_1(r,\phi) &= (\partial_x r)^2 - (\partial_x \phi)^2, \qquad \widetilde{N}_2(r) = \re^{2r}-2r-1, \qquad \widetilde{N}_3(r,\phi) = 2\left(\partial_x \phi\right)\left(\partial_x r\right).
\end{align*}
Hence, using~\eqref{eq:semdecomp} and~\eqref{e:nonldecomp} we rewrite the Duhamel formulation~\eqref{duhamelV} for $V$ as
\begin{align}
\begin{split}
\partial_x^m V(t) &= \re^{t L} \partial_x^m V_0 + \int_0^t S_c(t-s) \partial_x^m \begin{pmatrix} \widetilde{N}_1(r(s),\phi(s)) - \widetilde{N}_2(r(s)) \\ \partial_x \widetilde{N}_3(r(s),\phi(s)) \\ \gamma \partial_x^2 \widetilde{N}_2(r(s))\end{pmatrix} \de s\\
&\qquad + \, \int_0^t S_e(t-s) \partial_x^m \left(N_1(V(s)) + \partial_x N_2(V(s))\right) \de s,
\end{split}
\label{duhamelV0}
\end{align}
for $m = 0,1$. The nonlinearities in~\eqref{duhamelV0} obey the estimates~\eqref{eq:nonl1} and
\begin{align}
\begin{split}
\norm[L^\infty]{\widetilde{N}_1(r(s),\phi(s))} &\lesssim \norm[L^\infty]{\partial_x r(s)}^2 + \norm[L^\infty]{\partial_x \phi(s)}^2 \lesssim (1+s)^{-1+2\alpha} \eta_1(t)^2,\\
\norm[L^\infty]{\widetilde{N}_2(r(s))} &\lesssim \norm[L^\infty]{r(s)}^2 \lesssim \eta_1(t)^2,\\
\norm[L^\infty]{\widetilde{N}_3(r(s),\phi(s))} &\lesssim \norm[L^\infty]{\partial_x r(s)}\norm[L^\infty]{\partial_x \phi(s)} \lesssim (1+s)^{-1+\alpha} \eta_1(t)^2,\\
\norm[L^\infty]{\partial_x \widetilde{N}_2(r(s))} &\lesssim \norm[L^\infty]{r(s)}\norm[L^\infty]{\partial_x r(s)} \lesssim \frac{\eta_1(t)^2}{\sqrt{1+s}},\\
\norm[L^\infty]{\partial_x \widetilde{N}_3(r(s),\phi(s))} &\lesssim \norm[L^\infty]{\partial_x^2 r(s)}\norm[L^\infty]{\partial_x \phi(s)} + \norm[L^\infty]{\partial_x r(s)}\norm[L^\infty]{\partial_x^2 \phi(s)}\\ &\lesssim (1+s)^{-1+\alpha} \eta_1(t)(\eta_1(t)+\eta_2(t)),
\end{split} \label{eq:nonl3}
\end{align}
for all $s \in [0,t]$, where we use $\eta(t) \leq 1$. Thus, we bound~\eqref{duhamelV0} with the aid of Lemmas~\ref{lem:difest} and~\ref{lem:expest} and the estimates~\eqref{eq:nonl1} and~\eqref{eq:nonl3} and arrive at
\begin{align} \label{e:estV00}
\begin{split}
\norm[L^\infty]{V(t)} &\lesssim \norm[L^\infty]{V_0} + \int_0^t \frac{\norm[L^\infty]{\widetilde{N}_1(r(s),\phi(s))}}{1+t-s} \de s + \int_0^t \frac{\norm[L^\infty]{\widetilde{N}_2(r(s))}}{(1+t-s)^2} \de s\\
&\qquad + \, \int_0^t \frac{\norm[L^\infty]{\widetilde{N}_3(r(s),\phi(s))}}{\sqrt{1+t-s}} \de s + \int_0^t \left(\frac{\norm[L^\infty]{N_1(V(s))}}{\re^{\mu_0 (t-s)}} + \frac{\norm[L^\infty]{N_2(V(s))}}{\re^{\mu_0 (t-s)}\sqrt{t-s}}\right) \de s\\
&\lesssim \norm[L^\infty]{V_0} + \int_0^t \frac{\eta_1(t)^2}{(1+t-s)(1+s)^{1-2\alpha}} \de s + \int_0^t \frac{\eta_1(t)^2}{(1+t-s)^2} \de s\\
&\qquad + \, \int_0^t \frac{\eta_1(t)^2}{\sqrt{1+t-s}(1+s)^{1-\alpha}} \de s + \int_0^t \frac{\eta_1(t)^2}{\re^{\mu_0 (t-s)}} \left(1 + \frac{1}{\sqrt{1+s}\sqrt{t-s}}\right) \de s\\
&\lesssim \epsilon + \eta_1(t)^2,
\end{split}
\end{align}
and
\begin{align} \label{e:estV01}
\begin{split}
\norm[L^\infty]{\partial_x V(t)} &\lesssim \frac{\norm[W^{1,\infty}]{V_0}}{\sqrt{1+t}} + \int_0^t \frac{\norm[L^\infty]{\widetilde{N}_1(r(s),\phi(s))}}{\left(1+t-s\right)^{\frac{3}{2}}} \de s + \int_0^t \frac{\norm[L^\infty]{\partial_x \widetilde{N}_2(r(s))}}{\left(1+t-s\right)^2} \de s\\
&\quad \ + \, \int_0^t \frac{\norm[L^\infty]{\widetilde{N}_3(r(s),\phi(s))}}{1+t-s} \de s + \int_0^t \left(\frac{\norm[L^\infty]{\partial_x N_1(V(s))}}{\re^{\mu_0 (t-s)}} + \frac{\norm[L^\infty]{\partial_x N_2(V(s))}}{\re^{\mu_0 (t-s)}\sqrt{t-s}}\right) \de s\\
&\lesssim \frac{\norm[W^{1,\infty}]{V_0}}{\sqrt{1+t}} + \int_0^t \frac{\eta_1(t)^2}{(1+t-s)^{\frac{3}{2}}(1+s)^{1-2\alpha}} \de s + \int_0^t \frac{\eta_1(t)^2}{(1+t-s)^2 \sqrt{1+s}} \de s\\
&\quad \ + \, \int_0^t \frac{\eta_1(t)^2}{(1+t-s)(1+s)^{1-\alpha}} \de s + \int_0^t \frac{\eta_1(t)\left(\eta_1(t) + \eta_2(t)\right)}{\re^{\mu_0 (t-s)}\sqrt{1+s}} \left(1 + \frac{1}{\sqrt{t-s}}\right) \de s\\
&\lesssim \frac{\epsilon + \eta_1(t)\left(\eta_1(t) + \eta_2(t)\right)}{\sqrt{1+t}}.
\end{split}
\end{align}

Subsequently, we start bounding the phase variable $\phi$ and its derivatives, which satisfy the Duhamel formulation
\begin{align*}
\partial_x^m \phi(t) &= \re^{t \partial_x^2} \partial_x^m \phi_0 + \int_0^t \re^{(t-s) \partial_x^2} \partial_x^m \widetilde{N}_3(s) \de s,
\end{align*}
for $m = 0,1,2$. Thus, applying the standard semigroup estimates~\eqref{e:estphir} and the nonlinear estimate~\eqref{eq:nonl3}, we obtain
\begin{align} \label{e:phiest}
\begin{split}
\norm[L^\infty]{\phi(t)} &\lesssim \norm[L^\infty]{\phi_0} + \int_0^t \frac{\eta_1(t)^2}{(1+s)^{1-\alpha}} \de s \lesssim \left(\epsilon + \eta_1(t)^2\right)(1+t)^\alpha,\\
\norm[L^\infty]{\partial_x^m \phi(t)} &\lesssim \frac{\norm[W^{2,\infty}]{\phi_0}}{\sqrt{1+t}} + \int_0^t \frac{\eta_1(t)\left(\eta_1(t) + \eta_2(t)\right)}{\sqrt{t-s}(1+s)^{1-\alpha}} \de s \lesssim \frac{\epsilon + \eta_1(t)\left(\eta_1(t) + \eta_2(t)\right)}{(1+t)^{\frac{1}{2}-\alpha}},
\end{split}
\end{align}
for $m = 1,2$.

Combining the estimates~\eqref{e:estV00},~\eqref{e:estV01} and~\eqref{e:phiest} yields a constant $C_1 > 1$ such that
\begin{align}
\eta_1(t) \leq C_1\left(\epsilon + \eta_1(t)\left(\eta_1(t) + \eta_2(t)\right)\right). \label{e:etaest10}
\end{align}
On the other hand, employing the bound~\eqref{e:rest} on $\partial_x^2 r(t)$ we find a constant $C_2 > 1$ such that
\begin{align}
\eta_2(t) \leq C_2\left(\epsilon + \eta_1(t)\right). \label{e:etaest20}
\end{align}
Finally, as in the proof of Theorem~\ref{thm:explong} we combine~\eqref{e:etaest20} into~\eqref{e:etaest10} to obtain the key inequality~\eqref{e:etaest0} with $C = 6C_1C_2^2$.
\end{proof}

\section{Nonlinear stability against bounded perturbations in the real Ginzburg-Landau equation} \label{sec:GL}

In this section we establish nonlinear stability of the periodic roll solutions~\eqref{rolls} in the real Ginzburg-Landau equation~\eqref{GL} against $C_{\mathrm{ub}}^m$-perturbations. We exploit that by setting $\gamma = 0$ in the modified Ginzburg-Landau system~\eqref{eq:solsmodGL} we retrieve the real Ginzburg-Landau equation in the first component. Moreover, for $\gamma = 0$ the stability condition~\eqref{e:spec} reduces to the well-known Eckhaus condition $q^2 < \frac{1}{3}$. Hence, our nonlinear stability results in the real Ginzburg-Landau equation~\eqref{GL} are a direct consequence of the results established in~\S\ref{sec:mainresults}. Thus, inserting the perturbed periodic roll solution in polar form
\begin{align*}
	A(t,x) = \sqrt{1-q^2}\, \re^{\ri q x + r(x,t) + i\phi(x,t)},
\end{align*}
into the real Ginzburg-Landau equation~\eqref{GL}, we find the perturbation equation
\begin{align}
	\label{eq:GL}
\begin{split}
	\partial_t r &= \partial_x^2 r - 2q \partial_x \phi - \left(1-q^2\right) \left(\re^{2r} - 1\right) + (\partial_x r)^2 - (\partial_x \phi)^2, \\
	\partial_t \phi &= \partial_x^2 \phi + 2q \partial_x r + 2(\partial_x r)(\partial_x\phi),
\end{split}
\end{align}
for which we establish the following result.

\begin{theorem}\label{thm:stabGL}
	Let $q^2 < 1/3$.
	Then, there exist $M_0, \varepsilon_0 > 0$ such that for all $\varepsilon \in (0,\varepsilon_0)$ and $(r_0, \phi_0) \in C^2_{\mathrm{ub}}(\R,\R^2)$ satisfying
	\begin{align*}
		\norm[W^{2,\infty}]{r_0} + \norm[W^{2,\infty}]{\phi_0} < \varepsilon,
	\end{align*}
	there exists a global mild solution
	\begin{align*}
		(r,\phi) \in C\big([0,\infty),C_{\mathrm{ub}}^2(\R,\R^2)\big),
	\end{align*}
	of~\eqref{eq:GL} with initial condition $(r(0),\phi(0)) = (r_0,\phi_0)$ enjoying the estimates
	\begin{align*}
		\norm[L^\infty]{\phi(t)} \leq M_0 \varepsilon, \qquad \norm[L^\infty]{r(t)} + \norm[L^\infty]{\partial_x \phi(t)} \leq \frac{M_0 \varepsilon}{\sqrt{1+t}}, \qquad \norm[W^{1,\infty}]{\partial_x r(t)} + \norm[L^\infty]{\partial_x^2 \phi(t)} \leq \frac{M_0 \varepsilon}{1+t},
	\end{align*}
	for all $t \geq 0$.
\end{theorem}
\begin{proof}
	The statement is a direct application of Theorem~\ref{thm:partloc} after setting $\gamma = 0$, $p = 1$ and $B_0 = 0$. Here, we note that $B_0 = 0$ implies $B(t) = 0$ for all $t \geq 0$, since it holds $\gamma = 0$.
\end{proof}
 
The linearization of the perturbation equation~\eqref{eq:GL} has a relatively simple structure with one diffusive and one exponentially damped mode, which can be easily separated. The lowest-order terms in the expansion of the exponentially damped mode are captured by the variable
\begin{align} v(t) = r(t) + \frac{q}{1-q^2} \partial_x \phi(t), \label{v:exp}\end{align}
cf.~\eqref{r:exp}. Due to the linear exponential damping, the decay of $v(t)$ is dictated by the nonlinearities in~\eqref{eq:GL}, which are at least quadratic in $r(t)$ and $\partial_x \phi(t)$. Hence, we establish as a corollary of Theorem~\ref{thm:stabGL} that $v(t)$ decays at rate $t^{-1}$.

\begin{corollary} \label{cor:GL}
	Let $q^2 < 1/3$. Let $(r,\phi) \in C\big([0,\infty), C^2_{\mathrm{ub}}(\R,\R^2)\big)$ be the solution of~\eqref{eq:GL}, established in Theorem~\ref{thm:stabGL}. Then, there exists $\smash{\widetilde M_0} > 0$ such that the variable $v(t)$, defined by~\eqref{v:exp}, satisfies the estimate
	\begin{align*}
		\norm[L^\infty]{v(t)} \leq  \frac{\widetilde M_0 \varepsilon}{1+t},
	\end{align*}
	for all $t \geq 0$.
\end{corollary}
\begin{proof}
One readily observes that $v$ satisfies the equation
	\begin{align*}
		\partial_t v &= \partial_x^2 v - 2(1-q^2) v -(1-q^2) (\re^{2r} - 2r - 1) + (\partial_x r)^2 - (\partial_x \phi)^2 \\
		&\qquad - \dfrac{2q}{1-q^2}\left(q \partial_x^2 r + \partial_x((\partial_x r) (\partial_x \phi))\right).
	\end{align*}
	Thus, defining $L_0 = \partial_x^2 - 2(1-q^2)$ we find the Duhamel formula
	\begin{align}
		v(t) = \re^{t L_0} v(0) + \int_0^t \re^{(t-s)L_0} \curlN_v(r(s),\phi(s)) \de s,
		\label{eq:duhamelOptimalDecayGL}
	\end{align}
	with
	\begin{align*}
		\curlN_v(r, \phi) &= -\left(1-q^2\right) \left(\re^{2r} - 2r - 1\right) + (\partial_x r)^2 - (\partial_x \phi)^2 - \dfrac{2q}{1-q^2}\left(q \partial_x^2 r + \partial_x\left((\partial_x r) (\partial_x \phi\right))\right).
	\end{align*}
	The semigroup bound~\eqref{e:estphir} and the decay estimates provided by Theorem~\ref{thm:stabGL} yield
	\begin{align*}
		\norm[L^\infty]{v(t)} &\lesssim \re^{-2(1-q^2)t} \norm[L^\infty]{v(0)} \\
		&\qquad + \int_0^t \re^{-2(1-q^2)(t-s)}\left(\norm[L^\infty]{\partial_x^2 r(s)} + \left(\norm[W^{2,\infty}]{r(s)} + \norm[W^{1,\infty}]{\partial_x \phi(s)}\right)^2\right)\de s\\ &\lesssim \frac{\epsilon}{1+t},
	\end{align*}
	which proves the statement.
\end{proof}

\begin{remark}{\rm
Note that at $q = 0$ the $r$-component coincides to leading order with the exponentially damped mode, i.e.~we have $v(t) = r(t)$. So, Corollary~\ref{cor:GL} implies that the decay rate on the $r$-component in Theorem~\ref{thm:stabGL} can be improved by a factor $t^{-1/2}$ for $q = 0$.
}\end{remark}

\begin{remark}{ \rm
Although it is possible to identify a similar expansion as in~\eqref{v:exp} for the exponentially damped mode in the setting of the modified Ginzburg-Landau system~\eqref{eq:modGL}, we cannot expect that such a mode decays on the nonlinear level. Indeed, the nonlinearity in~\eqref{eq:pertModGLAbbreviation} contains terms which are quadratic in $r(t)$ and $\partial_x \phi(t)$, for which no decay can be expected, cf.~Theorem~\ref{thm:explong} and Remark~\ref{rem:optimality}. 
 }\end{remark}

\appendix

\section{Auxiliary results} \label{app:aux}

In order to establish $L^\infty$-bounds on the diffusive and exponentially damped part of the semigroup $\re^{tL}$ generated by the linearization $L$ of~\eqref{eq:pertModGLAbbreviation}, we need to bound convolution operators of the form
\begin{align}\left(S f\right)(x) = \int_{\R} \curlG(x-y) f(y) \de y, \label{eq:defS}\end{align}
cf.~\eqref{eq:semg}, where $\curlG(z)$ is defined through its Fourier transform by
\begin{align}\curlG(z) = \int_\R \curlA(k) \re^{\ri k z} \de k,\label{eq:defocs}\end{align}
cf.~\eqref{eq:greenG}, with $\curlA(k)$ being some exponentially localized smooth function, whose derivatives are also exponentially localized. Of course, Young's convolution inequality can be employed to bound $\norm[L^\infty]{Sf}$ by $\norm[L^q]{\curlG} \norm[L^p]{f}$ for $p,q \in [1,\infty]$ with $\frac{1}{p} + \frac{1}{q} = 1$. Since $\curlA$ belongs to the Schwartz class, so does $\curlG$. A precise estimate on the $L^q$-norm of $\curlG$ arises through integration by parts, where one exploits the oscillatory character of the integral~\eqref{eq:defocs}. We emphasize that taking $q = 1$ and $p = \infty$ leads to the pure $L^\infty$-estimate $\norm[L^\infty]{Sf} \leq \norm[L^1]{\curlG} \norm[L^\infty]{f}$, which is pivotal for the analysis in this paper, since we are working with \emph{bounded} perturbations. We note that the $L^p \to L^\infty$-estimate $\norm[L^\infty]{Sf} \leq \norm[L^q]{\curlG} \norm[L^p]{f}$ for $p \in [1,\infty)$ is only used in~\S\ref{sec:partloc}, where one considers bounded, partially localized perturbations.

In this appendix we take care of the technical estimates needed to obtain $L^\infty$-bounds on the components $S_c(t)$ and $S_e(t)$ of the semigroup $\re^{tL}$, cf.~\S\ref{sec:semdecomp}. That is, we derive $L^\infty$-bounds for convolution products of the form~\eqref{eq:defS}, where we assume that $\curlA$ has a specific Gaussian structure induced by the operator $L$, i.e.~it is of the form $\smash{\re^{t\Lambda(k)}}P(k)$, where $P(k)$ is some smooth matrix function exhibiting at most polynomial growth and $\Lambda(k)$ reflects the properties of the Fourier symbol of a second-order differential operator, cf.~Lemmas~\ref{lem:spec_cont} and~\ref{lem:specl}. Due to the Gaussian localization of $\curlA$, it is sufficient to integrate by parts twice in~\eqref{eq:defS} and thus only control derivatives of $\curlA$ up to second order.

We first establish the estimates needed to bound low frequencies, which corresponds to the case where $\curlA$ is compactly supported.

\begin{lemma}[Low-frequency estimates] \label{lem:semigroupEstimate1}
Let $p \in [1,\infty)$, $n,j \in \N_0$ and $m \in \N$ with $0 \leq j \leq m$. Let $k_0 > 0$. Let $\Lambda \colon (-k_0,k_0) \to \C^{m \times m}$ be given by
\begin{align*}\Lambda(k) = S(k)^{-1}\begin{pmatrix} \Lambda_c(k) & 0_{j \times (m-j)} \\ 0_{(m-j) \times j} & \Lambda_s(k) \end{pmatrix}S(k), \end{align*}
with $\Lambda_c \in C^2\big((-k_0,k_0),\C^{j \times j}\big), \Lambda_s \in C^2\big((-k_0,k_0),\C^{(m-j) \times (m-j)}\big)$ and $S \in C^2\big((-k_0,k_0),\operatorname{GL}_m(\C)\big)$ satisfying
\begin{itemize}
\item[i.] $\Lambda_c(0) = \Lambda_c'(0) = 0$;
\item[ii.] $\sup \Re \sigma(\Lambda_c(k)) < 0$ for $k \in (-k_0,k_0) \setminus \{0\}$;
\item[iii.] $\sup \Re \sigma(\Lambda_c''(0)) < 0$;
\item[iv.] $\sup \Re \sigma(\Lambda_s(k)) < 0$ for $k \in (-k_0,k_0)$.
\end{itemize}
Furthermore, let $P_s,P_c \colon \R \to \C^{m \times m}$ be given by
\begin{align*}P_c(k) &= S(k)^{-1} \begin{pmatrix} I_{j \times j} & 0_{j \times (m-j)} \\ 0_{(m-j) \times j} & 0_{(m-j) \times (m-j)} \end{pmatrix} S(k) \varpi(k),\\
P_s(k) &= S(k)^{-1} \begin{pmatrix} 0_{j \times j} & 0_{j \times (m-j)} \\ 0_{(m-j) \times j} & I_{(m-j) \times (m-j)} \end{pmatrix} S(k) \varpi(k),\end{align*}
where $\varpi \in C^2\big(\R,\C^{m \times m}\big)$ satisfies $\overline{\operatorname{supp}(\varpi)} \subset (-k_0,k_0)$. Then, there exists $\mu_0 > 0$ such that the estimates
\begin{align}
\norm[L^\infty]{\int_\R \int_\R \re^{t\Lambda(k)} k^n P_c(k) \re^{\ri k(\cdot-y)} \de k f(y) \de y} &\lesssim (1+t)^{-\frac{n}{2}} \norm[L^\infty]{f},
\label{eq:semigroupEstimate1}\\
\norm[L^\infty]{\int_\R \int_\R \re^{t\Lambda(k)} P_s(k) \re^{\ri k(\cdot-y)} \de k f(y) \de y} &\lesssim \re^{-\mu_0 t} \norm[L^\infty]{f},
\label{eq:semigroupEstimate2}
\end{align}
hold for all $f \in C_{\mathrm{ub}}(\R,\C^m)$ and $t \geq 0$.
Additionally, the estimate
\begin{align}
	\norm[L^\infty]{\int_\R \int_\R \re^{t\Lambda(k)} k^n P_c(k) \re^{\ri k(\cdot-y)} \de k f(y) \de y} \lesssim (1+t)^{-\frac{1}{2p} - \frac{n}{2}} \norm[L^p \cap L^\infty]{f},
	\label{eq:semigroupEstimateLocalized}
\end{align}
holds for all $f \in L^p(\R, \C^m) \cap C_{\mathrm{ub}}(\R, \C^m)$ and $t \geq 0$.
\end{lemma}
\begin{proof}
We start by obtaining estimates on the matrix exponential
\begin{align*} \re^{t \Lambda(k)} =  S(k)^{-1}\begin{pmatrix} \re^{t \Lambda_c(k)} & 0_{j \times (m-j)} \\ 0_{(m-j) \times j} & \re^{t \Lambda_s(k)} \end{pmatrix}S(k),\end{align*}
and its derivatives for $k \in J := \overline{\operatorname{supp}(\varpi)}$ and $t \geq 0$.

By Taylor's theorem, the assumptions i.-iii.~and the fact that $\Lambda_c$ is $C^2$, the matrix function $M \colon J \to \C^{j \times j}$ given by $M(k) = k^{-2} \Lambda_c(k)$ for $k \neq 0$ and $M(0) = \frac{1}{2}\Lambda_c''(0)$ is continuous and there exists $\mu_1 > 0$ such that
\begin{align}\sup \Re \sigma\left(M(k)\right) < -\mu_1, \qquad k \in J,\label{e:specbM}\end{align}
where we use that the eigenvalues of $M(k)$ depend continuously on $k$. To bound the matrix exponential $\re^{t M(k)}$ we collect some facts from~\cite[Chapter A-III, \S7]{ARE84}. First, since $J$ is compact and $M$ is continuous, the multiplication operator $A \colon f \mapsto M f$ generates a strongly continuous semigroup $(T(t))_{t \geq 0}$ on $C(J,\C^j)$, which is given by
$$(T(t)f)(k) = \re^{t M(k)} f(k), \qquad k \in J.$$
Second, the growth bound of the semigroup $(T(t))_{t \geq 0}$ coincides with the spectral bound of $A$. Third, the spectrum of $A$ is given by
$$\sigma(A) = \bigcup_{k \in J} \sigma(M(k)).$$
Combining the latter three observations with~\eqref{e:specbM} yields that the growth bound of the semigroup $(T(t))_{t \geq 0}$ is smaller than $-\mu_1$, which implies $\snorm{\re^{s M(k)}} \lesssim \re^{-\mu_1 s}$ for all $s \geq 0$ and $k \in J$. In particular, taking $k \in J$ and $s = k^2 t$ with $t \geq 0$ in the previous, we arrive at
\begin{align} \snorm{\re^{t \Lambda_c(k)}} \lesssim \re^{-\mu_1 k^2 t}, \qquad k \in J, \, t \geq 0. \label{e:supb1}\end{align}

Since $\Lambda_s$ is continuous, the eigenvalues of $\Lambda_s(k)$ depend continuously on $k$. Hence, as $J$ is compact, there exist $\mu_0, \mu_2 > 0$ such that
\begin{align}\sup \Re \sigma(\Lambda_s(k)) < -\mu_2 < -\mu_0, \qquad k \in J. \label{e:specbL}\end{align}
The bound on the matrix exponential $\re^{t \Lambda_s(k)}$ for $k \in J$ and $t \geq 0$ is obtained analogously as the bound on the $\re^{t M(k)}$. Thus, using~\eqref{e:specbL} and the fact that $\Lambda_s$ is continuous, we arrive at
\begin{align} \snorm{\re^{t \Lambda_s(k)}} \lesssim \re^{-\mu_2 t}, \qquad k \in J, \, t \geq 0.\label{e:supb3}\end{align}

Using the standard integral representation of the Fr\'echet derivative of the matrix exponential, we compute
\begin{align} \label{e:matrixexp}
\begin{split}
\partial_k \left(\re^{t \Lambda_i(k)}\right) &= t \int_0^1 \re^{l t \Lambda_i(k)} \Lambda_i'(k) \re^{(1-l) t \Lambda_i(k)} \de l,\\
\partial_k^2 \left(\re^{t \Lambda_i(k)}\right) &= t^2 \int_0^1 \int_0^1 l \re^{l \ell t \Lambda_i(k)} \Lambda_i'(k) \re^{l(1-\ell) t \Lambda_i(k)} \Lambda_i'(k) \re^{(1-l) t \Lambda_i(k)} \de \ell \de l\\
&\qquad + \, t^2 \int_0^1 \int_0^1 (1-l) \re^{l t \Lambda_i(k)} \Lambda_i'(k) \re^{(1-l) \ell t \Lambda_i(k)} \Lambda_i'(k) \re^{(1-\ell)(1-l) t \Lambda_i(k)} \de \ell \de l\\
&\qquad + \, t \int_0^1 \re^{l t \Lambda_i(k)} \Lambda_i''(k) \re^{(1-l) t \Lambda_i(k)} \de l,
\end{split}
\end{align}
for $i = s,c$. By the mean value theorem and assumption i.~it holds $\snorm{\Lambda_c'(k)} \lesssim |k|$. Thus, taking norms in the above expressions and using~\eqref{e:supb1} and $\Lambda_c \in C^2$ yields
\begin{align}
\begin{split}
\snorm{\partial_k \left(\re^{t \Lambda_c(k)}\right)} &\lesssim |k| t \re^{-\mu_1 k^2 t},\\
\snorm{\partial_k^2 \left(\re^{t \Lambda_c(k)}\right)} &\lesssim t\left(1 + k^2 t\right) \re^{-\mu_1 k^2 t},
\end{split}\label{e:supb2} \qquad k \in J, \, t \geq 0.
\end{align}
On the other hand, we use the estimate~\eqref{e:supb2} and the fact that $\Lambda_s \in C^2$ to bound the derivatives
\begin{align}
\begin{split}
\snorm{\partial_k \left(\re^{t \Lambda_s(k)}\right)} &\lesssim t \re^{-\mu_2 t},\\
\snorm{\partial_k^2 \left(\re^{t \Lambda_s(k)}\right)} &\lesssim t\left(1 + t\right) \re^{-\mu_2 t},
\end{split}\label{e:supb4} \qquad k \in J, \, t \geq 0.
\end{align}

All in all, observing
\begin{align} \re^{t \Lambda(k)} P_c(k) = S(k)^{-1}\begin{pmatrix} \re^{t \Lambda_c(k)} & 0_{j \times (m-j)} \\ 0_{(m-j) \times j} & 0_{(m-j) \times (m-j)} \end{pmatrix}S(k) \varpi(k), \label{e:matexp}\end{align}
and using~\eqref{e:supb1},~\eqref{e:supb2} and $\varpi, S \in C^2$, we arrive at
\begin{align}
\begin{split}
\snorm{\re^{t \Lambda(k)}P_c(k)} &\lesssim \re^{-\mu_1 k^2 t}, \\
\snorm{\partial_k \left(\re^{t \Lambda(k)}P_c(k)\right)} &\lesssim (1+|k| t)\re^{-\mu_1 k^2 t},\\
\snorm{\partial_k^2 \left(\re^{t \Lambda(k)}P_c(k)\right)} &\lesssim \left(1 + |k| t + t\left(1+k^2 t\right)\right)\re^{-\mu_1 k^2 t},
\end{split} \qquad k \in J, \, t \geq 0. \label{e:supb11}\end{align}
Similarly, using~\eqref{e:supb3},~\eqref{e:supb4} and $\varpi, S \in C^2$, we obtain
\begin{align}
\begin{split}
\snorm{\re^{t \Lambda(k)}P_s(k)} &\lesssim \re^{-\mu_2 t},\\
\snorm{\partial_k \left(\re^{t \Lambda(k)}P_s(k)\right)} &\lesssim (1+t)\re^{-\mu_2 t},\\
\snorm{\partial_k^2 \left(\re^{t \Lambda(k)}P_s(k)\right)} &\lesssim \left(1+ t\right)^2 \re^{-\mu_2 t},
\end{split} \qquad k \in J, \, t \geq 0. \label{e:supb21}\end{align}

Having obtained suitable bounds on the matrix exponential $\re^{t \Lambda(k)}$, we proceed by establishing the estimate~\eqref{eq:semigroupEstimate1}. We distinguish between the cases $t \in [0,1]$ and $t > 1$. For the case $t \in [0,1]$, we rewrite
\begin{align}
\int_\R \re^{t \Lambda(k)} k^n P_c(k) \re^{\ri kz} \de k
&= \dfrac{1}{1+z^2} \left(\int_\R \re^{t \Lambda(k)} k^n P_c(k)  \re^{\ri kz} \de k + \int_\R \re^{t\Lambda(k)} k^n P_c(k) z^2 \re^{\ri kz} \de k\right), \label{eq:split1}
\end{align}
with $z \in \R$. Using $J \subset (-k_0,k_0)$ and~\eqref{e:supb11} we conclude that the first integral in~\eqref{eq:split1} is bounded uniformly for $z \in \R$. For the second integral we exploit $z^2 \re^{\ri kz} = -\partial_k^2 \re^{\ri kz}$. Thus, we use integration by parts to obtain
\begin{align}
\begin{split}
-\int_\R \re^{t\Lambda(k)} k^n P_c(k) z^2 \re^{\ri kz} \de k &= \int_\R \partial_k^2 \left(k^n \re^{t\Lambda(k)} P_c(k)\right) \re^{\ri kz} \de k\\
&= \int_\R \left(k^n \partial_k^2 \left(\re^{t\Lambda(k)} P_c(k)\right) + 2 nk^{n-1} \partial_k \left(\re^{t\Lambda(k)} P_c(k)\right)\right.\\
&\qquad\qquad \left. + \, n(n-1)k^{n-2} \re^{t\Lambda(k)} P_c(k) \right) \re^{\ri kz} \de k,
\end{split} \label{eq:split2}
\end{align}
which is bounded uniformly for $z \in \R$ by~\eqref{e:supb11} noting that $t \leq 1$ and $P_c \in C^2$ has compact support. Therefore, for all $t \in [0,1]$ and $f \in C_{\mathrm{ub}}(\R,\C^m)$ we obtain
	\begin{align*}
		\norm[L^\infty]{\int_\R \int_\R \re^{t\Lambda(k)} k^n P_c(k) \re^{\ri k(\cdot-y)} \de k f(y) \de y} &\lesssim \norm[L^\infty]{f} \int_\R (1+z^2)^{-1} \de z \lesssim (1+t)^{-\frac{n}{2}} \norm[L^\infty]{f}.
	\end{align*}
	
Now take $t > 1$. Then, we rewrite
\begin{align}
\int_\R \re^{t \Lambda(k)} k^n P_c(k) \re^{\ri kz} \de k &= \dfrac{1}{1+\frac{z^2}{t}} \left(\int_\R \re^{t\Lambda(k)} k^n P_c(k) \re^{\ri kz} \de k + \int_\R \re^{t\Lambda(k)} k^n P_c(k) \dfrac{z^2}{t} \re^{\ri kz} \de k\right), \label{eq:split3}
\end{align}
for $z \in \R$. To bound the first integral in~\eqref{eq:split3} we use~\eqref{e:supb11} and obtain
\begin{align*}
\snorm{\int_\R \re^{t\Lambda(k)} k^n P_c(k) \re^{\ri kz} \de k} \lesssim \int_{J} \snorm{k^n \re^{t \Lambda(k)}} \de k \lesssim \int_\R |k|^n \re^{-\mu_1 k^2 t} \de k \lesssim t^{-\frac{n+1}{2}},
\end{align*}
for $z \in \R$. To bound the second integral in~\eqref{eq:split3}, we again exploit~\eqref{eq:split2}. Thus, using~\eqref{e:supb11} and noting that $t > 1$ and $P_c \in C^2$ has compact support, we obtain
\begin{align*}
&\snorm{\int_\R \re^{t\Lambda(k)} k^n P_c(k) \dfrac{z^2}{t} \re^{\ri kz} \de k} = \frac{1}{t} \snorm{\int_{J} \partial_k^2 \left(k^n \re^{t\Lambda(k)} P_c(k)\right) \re^{\ri kz} \de k}\\ &\quad \lesssim \int_\R \re^{-\mu_1 k^2 t} \left(|k|^n \left(t^{-1} + |k| + \left(1+k^2 t\right)\right) + n|k|^{n-1} (t^{-1} +|k|) + n(n-1)|k|^{n-2} t^{-1} \right) \de k \lesssim t^{-\frac{n+1}{2}},
\end{align*}
for $z \in \R$. Therefore, we derive for all $t > 1$ and $f \in C_{\mathrm{ub}}(\R,\C^m)$
\begin{align*}
\norm[L^\infty]{\int_\R \int_\R \re^{t\Lambda(k)} k^n P_c(k) \re^{\ri k(\cdot-y)} \de k f(y) \de y} &\lesssim t^{-\frac{n+1}{2}} \norm[L^\infty]{f} \int_\R \left(1 + \dfrac{z^2}{t}\right)^{-1} \de z\\ &\lesssim (1+t)^{-\frac{n}{2}} \norm[L^\infty]{f},
\end{align*}
which proves~\eqref{eq:semigroupEstimate1}.

To establish the estimate~\eqref{eq:semigroupEstimate2} we rewrite
\begin{align*}
\int_\R \re^{t \Lambda(k)} P_s(k) \re^{\ri kz} \de k
= \dfrac{1}{1+z^2} \left(\int_\R \re^{t \Lambda(k)} P_s(k) \de k - \int_\R \partial_k^2 \left(\re^{t\Lambda(k)} P_s(k)\right) \re^{\ri kz} \de k\right),
\end{align*}
for $t \geq 0$ and $z \in \R$. Then, using $J \subset (-k_0,k_0)$,~\eqref{e:supb21} and the fact that $P_s \in C^2$ we arrive at the bound
\begin{align*}
\snorm{\int_\R \re^{t \Lambda(k)} P_s(k) \re^{\ri kz} \de k} \lesssim \dfrac{(1+t)^2\re^{-\mu_2 t}}{1+z^2}, \qquad t \geq 0, \, z \in \R.
\end{align*}
Therefore, for all $t \geq 0$ and $f \in C_{\mathrm{ub}}(\R,\C^m)$ we obtain
	\begin{align*}
		\norm[L^\infty]{\int_\R \int_\R \re^{t\Lambda(k)} P_s(k) \re^{\ri k(\cdot-y)} \de k f(y) \de y} &\lesssim \norm[L^\infty]{f} \int_\R \frac{(1+t)^2\re^{-\mu_2 t}}{1+z^2} \de z \lesssim \re^{-\mu_0 t}\norm[L^\infty]{f},
	\end{align*}
where we use \eqref{e:specbL} in the last estimate.
This completes the proof of~\eqref{eq:semigroupEstimate2}.

Finally, we establish the estimate~\eqref{eq:semigroupEstimateLocalized}.
If $t \in [0,1]$ the estimate~\eqref{eq:semigroupEstimate1} implies~\eqref{eq:semigroupEstimateLocalized}.
Therefore, let $t > 1$. We denote by $q \in (1,\infty]$ the H\"older conjugate of $p$ so that $\frac{1}{p}+\frac{1}{q} = 1$.
Then, Young's convolution inequality yields
\begin{align*}
	\norm[L^\infty]{\int_\R \int_\R \re^{t\Lambda(k)} k^n P_c(k) \re^{\ri k \left(\cdot - y\right)} \de k f(y) \de y} &\lesssim \norm[q]{\int_\R \re^{t\Lambda(k)} k^n P_c(k) \re^{\ri k \cdot} \de k} \norm[p]{f} \\
	& \lesssim \norm[p]{k \mapsto \re^{t\Lambda(k)} k^n P_c(k)} \norm[p]{f} \\
	& \lesssim (1+t)^{-\frac{n}{2}-\frac{1}{2p}} \norm[p]{f},
\end{align*}
where we used that the Fourier transform is an $L^p$-$L^q$-isomorphism in the second inequality and the estimate~\eqref{e:supb11} in the last inequality.
Combining this with the estimate for $t \in [0,1]$ completes the proof of~\eqref{eq:semigroupEstimateLocalized} and of the Lemma.
\end{proof}

Following the same strategy we also obtain the corresponding high-frequency result.

\begin{lemma}[High-frequency estimates] \label{lem:semigroupEstimate3}
Let $m \in \N$ and $\Lambda \in C^2\big(\R,\C^{m \times m}\big)$. Assume
\begin{itemize}
\item[i.] $\sup \Re \sigma(\Lambda(k)) < 0$ for all $k \in \R$;
\item[ii.] $M_0 = \displaystyle \lim_{\ell \to 0} \ell^2 \Lambda\left(\ell^{-1}\right)$ exists and satisfies $\sup \Re \sigma(M_0) < 0$;
\item[iii.] $\snorm{\Lambda'(k)} \lesssim 1+|k|$ and $\snorm{\Lambda''(k)} \lesssim 1$ for $k \in \R$.
\end{itemize}
Furthermore, let $P_0 \in \C^{m\times m}$. Then, there exists $\mu_0 > 0$ such that the estimate
\begin{align}
\norm[L^\infty]{\int_\R \int_{\R} \re^{\Lambda(k) t} k^n P_0 \re^{\ri k(\cdot-y)} \de k f(y) \de y} \lesssim \re^{-\mu_0 t} \left(1 + t^{-\frac{n}{2}}\right) \norm[L^\infty]{f},
\label{eq:semigroupEstimate3}
\end{align}
holds for all $f \in C_{\mathrm{ub}}(\R,\C^m)$, $t > 0$ and $n = 0,1$.
\end{lemma}
\begin{proof}
We start by obtaining estimates on the matrix exponential $\re^{t \Lambda(k)}$ and its derivatives for $k \in \R$ and $t \geq 0$. By assumption ii.~and the fact that $\Lambda$ is continuous, there exists $\ell_0 > 0$ such that the matrix function $M \colon [-\ell_0,\ell_0] \to \C^{m \times m}$ given by $M(\ell) = \ell^2 \Lambda(\ell^{-1})$ for $\ell \neq 0$ and $M(0) = M_0$ is continuous and there exists $\mu_1 > 0$ such that
\begin{align}\sup \Re \sigma\left(M(\ell)\right) < -\mu_1, \qquad \ell \in [-\ell_0,\ell_0].\label{e:specbM2}\end{align}
Hence, as in the proof of Lemma~\ref{lem:semigroupEstimate1}, the estimate
\begin{align*} \snorm{\re^{s M(\ell)}} \lesssim \re^{-\mu_1 s}, \qquad \ell \in [-\ell_0,\ell_0], \, s \geq 0.\end{align*}
follows by~\eqref{e:specbM2} and the fact that $M$ is continuous. Thus, taking $\ell = k^{-1}$ and $s = k^2 t$ with $t \geq 0$ and $k \in \R \setminus (-\ell_0^{-1},\ell_0^{-1})$ in the previous yields
\begin{align*} \snorm{\re^{t \Lambda(k)}} \lesssim \re^{-\mu_1 k^2 t}, \qquad k \in \R \setminus \left(-\ell_0^{-1},\ell_0^{-1}\right), \, t \geq 0.\end{align*}
Similarly, by continuity of $\Lambda$ and assumption i., there exists $\widetilde{\mu}_1 > 0$ such that
\begin{align*} \snorm{\re^{t \Lambda(k)}} \lesssim \re^{-\widetilde{\mu}_1 t}, \qquad k \in \left[-\ell_0^{-1},\ell_0^{-1}\right], \, t \geq 0.\end{align*}
Combining the latter two estimates yields $\mu_0,\mu_2 > 0$ such that
\begin{align} \snorm{\re^{t \Lambda(k)}} \lesssim \re^{-\mu_0 t - \mu_2 k^2 t}, \qquad k \in \R, \, t \geq 0. \label{e:supb5}\end{align}

Bounds on the derivatives of $\re^{t \Lambda(k)}$ follow by employing the formulas~\eqref{e:matrixexp}, the estimate~\eqref{e:supb5} and assumption iii. All in all, we obtain
\begin{align}
\begin{split}
\snorm{\partial_k \left(\re^{t \Lambda(k)}\right)} &\lesssim t (1+|k|) \re^{-\mu_0 t - \mu_2 k^2 t},\\
\snorm{\partial_k^2 \left(\re^{t \Lambda(k)}\right)} &\lesssim t\left(1 + (1+|k|)^2 t\right) \re^{-\mu_0 t - \mu_2 k^2 t},
\end{split}\label{e:supb6} \qquad k \in \R, \, t \geq 0.
\end{align}

We are now in the position to establish the estimate~\eqref{eq:semigroupEstimate3}. We proceed as in the proof of Lemma~\ref{lem:semigroupEstimate1} for the case $t > 1$. That is, we rewrite
\begin{align*}
\int_{\R} \re^{t \Lambda(k)} k^n P_0 \re^{\ri kz} \de k
&= \left(\int_{\R} \re^{t \Lambda(k)} k^n P_0 \re^{\ri kz} \de k - \frac{2n}{t}\int_{\R} \partial_k \left(\re^{t\Lambda(k)}\right) P_0 \re^{\ri kz} \de k\right.\\ &\qquad\qquad \left.-\,\frac{1}{t}\int_{\R} \partial_k^2 \left(\re^{t\Lambda(k)}\right) P_0 k^n \re^{\ri kz} \de k\right) \left(1+\frac{z^2}{t}\right)^{-1},
\end{align*}
for $n = 0,1$, $t \geq 0$ and $z \in \R$. We bound the integrals on the right-hand side in the latter one by one using~\eqref{e:supb5} and~\eqref{e:supb6}. Thus, we obtain
\begin{align*}
\snorm{\int_{\R} \re^{t \Lambda(k)} k^n P_0 \re^{\ri kz} \de k} &\lesssim \re^{-\mu_0 t} \int_\R |k|^n \re^{-\mu_2 k^2 t} \de k \lesssim t^{-\frac{n+1}{2}} \re^{-\mu_0 t},\\
\snorm{\frac{n}{t}\int_{\R} \partial_k \left(\re^{t\Lambda(k)}\right) P_0 \re^{\ri kz} \de k} &\lesssim n\re^{-\mu_0 t} \int_\R (1+|k|) \re^{-\mu_2 k^2 t} \de k \lesssim \left(1+t^{-\frac{n}{2}}\right) t^{-\frac{1}{2}} \re^{-\mu_0 t},\\
\snorm{\frac{1}{t}\int_{\R} \partial_k^2 \left(\re^{t\Lambda(k)}\right) P_0 k^n \re^{\ri kz} \de k} &\lesssim \re^{-\mu_0 t} \int_\R \left(1 + (1+|k|)^2 t\right) |k|^n \re^{-\mu_2 k^2 t} \de k \lesssim t^{-\frac{n+1}{2}} \re^{-\mu_0 t},
\end{align*}
for $n = 0,1$, $t > 0$ and $z \in \R$. Therefore, we establish
\begin{align*}
\norm[L^\infty]{\int_\R \int_{\R} \re^{t\Lambda(k)} k^n P_0 \re^{\ri k(\cdot-y)} \de k f(y) \de y} &\lesssim \frac{\left(1 + t^{-\frac{n}{2}}\right)}{\sqrt{t}\, \re^{\mu_0 t}} \norm[L^\infty]{f} \int_\R \left(1 + \dfrac{z^2}{t}\right)^{-1} \de z\\ &\lesssim \re^{-\mu_0 t} \left(1 + t^{-\frac{n}{2}}\right) \norm[L^\infty]{f},
\end{align*}
for $n = 0,1$, $t > 0$ and $f \in C_{\mathrm{ub}}(\R,\C^m)$, which completes the proof.
\end{proof}

\bibliographystyle{abbrv}
\bibliography{glbib}

\end{document}